\documentclass[reqno]{amsart}

\usepackage[applemac]{inputenc}

\usepackage{amssymb}
\usepackage{graphicx}
\usepackage[cmtip,all]{xy}
\usepackage{verbatim}
\usepackage{tikz-cd}
\usepackage{mathtools}
\usepackage{graphicx}
\usepackage{mathrsfs}
\usepackage{geometry}
\usepackage{hyperref}

\DeclareMathAlphabet{\mathpzc}{OT1}{pzc}{m}{it}

\newtheorem{theorem}{Theorem}[section]
\newtheorem{lemma}[theorem]{Lemma}

\theoremstyle{definition}
\newtheorem{definition}[theorem]{Definition}

\theoremstyle{remark}

\numberwithin{equation}{section}

 \newcommand{\virgolette}{``}

\newcommand{\set}[1]{\ensuremath{\mathbb{#1}}}

\newcommand{\slantone}[2]{{\raisebox{.1em}{$#1$}\left/\raisebox{-.1em}{$#2$}\right.}}
\newcommand*{\defeq}{\mathrel{\vcenter{\baselineskip0.5ex \lineskiplimit0pt
                     \hbox{\scriptsize.}\hbox{\scriptsize.}}}%
                     =}

\newcommand{\slanttwo}[2]{{\raisebox{.2em}{$#1$} \big/ \raisebox{-.2em}{$#2$}}}

\newcommand\asim{\mathrel{%
  \ooalign{\raise0.1ex\hbox{$\sim$}\cr\hidewidth\raise-0.8ex\hbox{\scalebox{0.9}{$\scriptstyle{x}$}}\hidewidth\cr}}}

\newcommand{\proj}[1]{\ensuremath{\mathbb{P}^{#1}}}

\newcommand{\mani}{\ensuremath{\mathpzc{M}}}
\newcommand{\manir}{\ensuremath{\mathpzc{M}_{red}}}
\newcommand{\stsheaf}{\ensuremath{\mathcal{O}_{\mathpzc{M}}}}
\newcommand{\stsheafred}{\ensuremath{\mathcal{O}_{\mathpzc{M}_{red}}}}

\newcommand{\beq}{\begin{equation}}
\newcommand{\eeq}{\end{equation}}
\newcommand{\bear}{\begin{eqnarray}}
\newcommand{\eear}{\end{eqnarray}}
\newcommand{\longhookrightarrow}{\ensuremath{\lhook\joinrel\relbar\joinrel\rightarrow}}

\DeclareMathOperator{\coker}{coker}

\begin{document}

\begin{flushright}
DISIT-2018
\par\end{flushright}

\title{Non-Projected Supermanifolds and Embeddings in Super Grassmannians}


\author{Simone Noja}
\address{Dipartimento di Scienze e Innovazione Tecnologica - Università del Piemonte Orientale, 
Via T. Michel 11, 15121, Alessandria, Italy.}
\address{INFN - Sezione di Torino, via P. Giuria 1, 10125 Torino.}
\email{sjmonoja87@gmail.com}




\begin{abstract}
In this paper we give a brief account of the relations between non-projected supermanifolds and projectivity in supergeometry. Following the general results of \href{https://arxiv.org/abs/1706.01354}{arXiv:1706.01354}, we study an explicit example of non-projected and non-projective supermanifold over the projective plane and we show how to embed it into a super Grassmannian. The geometry of super Grassmannians is also reviewed in details.
\end{abstract}

\maketitle

\tableofcontents

\section{Introduction: Projectivity and Non-Projectivity in Supergeometry}

\noindent The problem of projectivity in supergeometry is a long-standing one. Indeed, large classes of complex supermanifolds whose reduced complex manifolds $\manir $ are projective - \emph{i.e.}\ there exists an embedding $\manir \hookrightarrow \proj n$ - are known to be \emph{non} {superprojective} (henceforth, projective), that is they do not admit an embedding $\mani \hookrightarrow \proj {n|m}$ for some projective superspace $\proj {n|m}.$ This is the case, for example, of a large class of complex super Grassmannians (see \cite{Manin} and section 4 of this paper). \\
The problem of projectivity is related to another central problem characterizing the theory of complex supermanifold, that of the so-called \emph{non-projected supermanifolds}: these are complex supermanifolds that do \emph{not} possess a projection to their reduced manifold $\mani \rightarrow \manir$. Indeed, it has been shown that any \emph{projected} supermanifold whose reduced manifold is projective, is also superprojective. In other words, if $\manir$ is a projective complex manifolds and $\mani $ is projected, the embedding $\manir \hookrightarrow \proj n$ can be lifted to an embedding of supermanifolds $\mani \hookrightarrow \proj {n|m}$ (see for example \cite{BPW}). Notice that, for this to be true, the existence of the projection map $\mani \rightarrow \manir$ is crucial: indeed if we let $\mathcal{L}_{red}$ be a very-ample line bundle on $\manir$, then $\pi^\ast \mathcal{L}_{red}$ will be very-ample on $\mani$, in the sense that $\pi^\ast \mathcal{L}_{red}$ will allow for the embedding at the level of the supermanifolds $\mani \hookrightarrow \proj {n|m}$ \cite{BPW}, \cite{NojaPhD}. \\
The story is different whenever a supermanifold is non-projected. The obstruction theory to find an embedding into projective superspace for a complex supermanifold has been studied for example in \cite{BPW}, back in the early days of supergeometry. There it is shown that the obstruction to extend the embedding map $\manir \hookrightarrow \proj n$ at the level of the reduced complex manifolds, to an embedding $\manir \hookrightarrow \proj {n|m}$ at the level of complex supermanifolds lies in in the cohomology groups $H^2 (Sym^{2k} \mathcal{F}_\mani)$ for $k=1, \ldots, \mbox{rank} \, \mathcal{F}_\mani/2$ and where the vector bundle $\mathcal{F}_\mani = \mathcal{J}_\mani / \mathcal{J}^2_\mani$ is constructed via a suitable quotient of the \emph{nilpotent bundle} $\mathcal{J}_\mani$ of the supermanifold, encoding the behavior of the anti-commutative nilpotent part of the geometry, see \cite{Manin}, \cite{NojaPhD}. This result has some obvious, yet remarkable, consequences: for example, by dimensional reasons, one sees that any supercurve, \emph{i.e.}\ any supermanifold of dimension $1|m$ constructed over a projective curve, is actually projective, and the issues regarding projectivity start arising for dimension $n|m$, for $n,m\geq 2.$ \\
Following these considerations, whist the literature fully acknowledged that in the realm of supergeometry projective superspaces $\proj {n|m}$ are not as important as they are in ordinary complex algebraic geometry, nothing has been said, by the way, about which sort of space is to be considered when one looks for a \emph{universal embedding space} for complex supermanifolds. In the recent \cite{CNR} this problem has been taken on starting from dimension $2|2$, working over the projective plane $\proj 2$, and it has been shown that a large class of non-projected complex supermanifolds does not indeed admit projective embeddings, while all of these non-projected \emph{and} non-projective supermanifolds admit embeddings in some complex \emph{super Grassmannians}, thus hinting that the same might happen also in higher dimensions. 

In the paper we consider again the problem of embedding a supermanifold into a super Grassmannians, enriching and clarifying the abstract results of \cite{CNR} by very explicit constructions and examples. In particular, in the first section of the paper the key concepts of supergeometry are revised and the notation is fixed, also the main result of \cite{CNR} are reported and put in context as to make the paper self-consistent. Next, following \cite{Manin}, the supergeometry of complex super Grassmannians is explained. In the last section it is shown how to build maps to super Grassmannians and the example of the $2|2$ dimensional supermanifold over $\proj 2$ characterized by a decomposable fermionic bundle $\mathcal{F}_\mani = \Pi \mathcal{O}_{\proj 2}(-1) \oplus \Pi \mathcal{O}_{\proj 2} (-2)$ is carried out in full details.\\

The interested reader might find further general references about supergeometry can be found in \cite{Manin}, \cite{ManinNC} and \cite{Vara}. On the problem of projectivity in supergeometry, the reader might refer to \cite{BPW}, \cite{PenSko}, and the recent \cite{NojaPi}, \cite{FLLN}, \cite{Bet}.

\section{Basics of Supermanifolds}

\noindent In this section we recall the basic definitions in the theory of (complex) supermanifolds. The interested reader might find more details in \cite{Manin} or \cite{NojaPhD}, which we will follow closely.
The most important notion in supergeometry is the one of superspace, which is defined as follows.
\begin{definition}[Superspace] A superspace is a pair $(|\mathpzc{M}|, \mathcal{O}_{\mathpzc{M}})$, where $|\mathpzc{M}|$ is a topological space and $\mathpzc{O}_{\mathpzc{M}}$ is a sheaf of $\mathbb{Z}_2$-graded supercommutative rings (super rings for short) defined over $|\mathpzc{M}|$ and such that the stalks $\mathcal{O}_{\mathpzc{M}, x}$ at every point of $|\mathpzc{M}|$ are {local rings}. \\
In other words, a superspace is a locally ringed space having structure sheaf given by a sheaf of super rings.
\end{definition}

\noindent The requirement about the stalks being local rings is the same thing as asking that the \emph{even} component of the stalk is a usual commutative local ring, for in superalgebra one has that if $A = A_0 \oplus A_1$ a super ring, then $A$ is local if and only if its even part $A_0$ is (see for example \cite{Vara}). \\
It is important to observe that one can always construct a superspace out of two classical data: a topological space, call it again $|\mathpzc{M}|$, and a vector bundle over $|\mathpzc{M}|$, call it $\mathcal{E}$ (analogously: a locally-free sheaf of $\mathcal{O}_{|\mathpzc{M}|}$-modules). Now, we denote $\mathcal{O}_{|\mani|}$ the sheaf of continuous functions (with respect to the given topology) on $|\mani|$ and we put $\bigwedge^0 \mathcal{E}^\ast = \mathcal{O}_{|\mani|}$. The sheaf of sections of the bundle of exterior algebras $\mathcal{\bigwedge^\bullet} \mathcal{E}^\ast$ has an obvious $\set{Z}_2$-grading (by taking its natural $\set{Z}$-grading $\mbox{mod}\,2$) and therefore in order to realise a superspace it is enough to take the structure sheaf $\mathcal{O}_\mathpzc{M}$ of the superspace to be the sheaf of sections valued in $\mathcal{O}_{|\mani|}$ of the bundle of exterior algebras. This is what is called \emph{local model}.
\begin{definition}[Local Model $\mathfrak{S}(|\mathpzc{M}|, \mathcal{E})$] Given a pair $(|\mathpzc{M}|, \mathcal{E})$, where $|\mathpzc{M}|$ is a topological space and $\mathcal{E}$ is a vector bundle over $|\mathpzc{M}|$, we call $\mathfrak{S}(|\mathpzc{M}|, \mathcal{E})$ the superspace modelled on the pair $(|\mathpzc{M}|, \mathcal{E})$, where the structure sheaf is given by the $\mathcal{O}_{|\mani|}$-valued sections of the exterior algebra $\bigwedge^\bullet \mathcal{E}^\ast$. 
\end{definition}  
\noindent This is a \emph{minimal} definition of local model: we have let $|\mani|$ to be no more than a topological space and as such we are only allowed to take $\mathcal{O}_{|\mani|}$ to be the sheaf of continuous functions on it. One can obviously work in a richer and more structured category, such as the \emph{differentiable}, \emph{complex analytic} or \emph{algebraic} category: from now on, we will work in the complex analytic category and we consider local models based on the pair $(\manir, \mathcal{E}),$ where $\manir$ is a \emph{complex manifold} (its underlying topological space will be denoted with $|\mani|$ and the sheaf of holomorphic functions on $\manir$ with $\stsheafred$) and where $\mathcal{E}$ is a \emph{holomorphic vector bundle} on $\manir$. We will call \emph{holomorphic local model} a local model constructed on these kind of data .\\ 
The concept of local model enters in the definition of the main character of this paper.
\begin{definition}[Complex Supermanifold] A complex supermanifold $\mani $ of dimension $n|m$ is a superspace that is \emph{locally} isomorphic to some holomorphic local model $\mathfrak{S} ( \manir, \mathcal{E})$, where $\manir$ is a complex manifold of dimension $n$ and $\mathcal{E}$ is a holomorphic vector bundle of rank $m$.
\end{definition}
\noindent In other words, if the topological space $|\mathpzc{M}|$ underlying $\manir $ has a basis $\{{U}_i \}_{i \in I}$, the structure sheaf $\stsheaf = \mathcal{O}_{\mani, 0} \oplus \mathcal{O}_{\mani, 1}$ of the supermanifold $\mani$ is described via a collection $\{ \psi_{{U}_i} \}_{i\in I}$ of local isomorphisms of sheaves 
\bear
U_i \longmapsto \psi_{{U}_i} :  \mathcal{O}_{\mathpzc{M}}\lfloor_{U_i} \stackrel{\cong}{\longrightarrow} \bigwedge^\bullet \mathcal{E}^\ast \lfloor_{U_i} 
\eear 
where we have denoted with $\bigwedge^\bullet \mathcal{E}^\ast$ the sheaf of sections of the exterior algebra of $\mathcal{E}$ considered with its $\mathbb{Z}_2$-gradation. \\
In general, given two superspaces we can define a morphism relating these two. 
\begin{definition}[Morphisms of Superspaces] Given two superspaces $\mathpzc{M}$ and $\mathpzc{N}$ a morphism $\varphi : \mathpzc{M} \rightarrow \mathpzc{N}$ is a pair $\varphi \defeq (\phi, \phi^\sharp)$ where 
\begin{enumerate}
\item
$\phi : |\mathpzc{M}| \rightarrow |\mathpzc{N}|$ is a continuous map of topological spaces; 
\item $\phi^\sharp : \mathcal{O}_{\mathpzc{N}} \rightarrow \phi_* \mathcal{O}_\mathpzc{M}$ is a morphism of sheaves of $\set{Z}_2$-graded rings, having the property that it preserves the $\set{Z}_2$-grading and that given any point $x\in |\mathpzc{M}|$, the homomorphism
$
\phi^\sharp_x : \mathcal{O}_{\mathpzc{N}, \phi (x)} \rightarrow \mathcal{O}_{\mathpzc{M}, x}
$
is local, that is it preserves the (unique) maximal ideal, $\phi^\sharp_x (\mathfrak{m}_{\phi (x)}) \subseteq \mathfrak{m}_x.$
\end{enumerate}
\end{definition}
\noindent This definition applies in particular to the case of complex supermanifolds and enters the definition of \emph{sub-supermanifolds}. Indeed, as in the ordinary theory, a \emph{sub-supermanifold} is defined in general as a pair $(\mathpzc{N}, \iota)$, were $\mathpzc{N}$ is a supermanifold and $\iota \defeq (\iota , \iota^\sharp) : ( \mathpzc{N}, \mathcal{O}_{\mathpzc{N}} ) \rightarrow (\mani, \stsheaf)$ is an \emph{injective} morphism with some regularity property. In particular, depending on these regularity properties, we can distinguish between two kind of sub-supermanifolds. We start from the milder notion.
\begin{definition}[Immersed Supermanifold] Let $\iota \defeq (i, i^\sharp ) : (|\mathpzc{N}|, \mathcal{O}_{\mathpzc{N}} ) \rightarrow (|\mani|, \stsheaf)$ be a morphism of supermanifolds. We say that $(\mathpzc{N}, \iota)$ is an immersed supermanifold if $i : |\mathpzc{N}| \rightarrow |\mani|$ is injective and the differential $(d\iota)(x) : \mathcal{T}_\mathpzc{N} (x) \rightarrow \mathcal{T}_\mani {(i (x))} $ is injective for all $x \in |\mathpzc{N}|.$ 
\end{definition} 
\noindent Making stronger requests, we can give instead the following definition.
\begin{definition}[Embedded Supermanifold] Let $\iota \defeq (i, i^\sharp ) : (|\mathpzc{N}|, \mathcal{O}_{\mathpzc{N}} ) \rightarrow (|\mani|, \stsheaf)$ be a morphism of supermanifolds. We say that $(\mathpzc{N}, \iota)$ is an embedded supermanifold if it is an immersed submanifold and $i: |\mani| \rightarrow |\mathpzc{N}|$ is an homeomorphism onto its image. \\
In particular, if $\iota (|\mathpzc{N}|) \subset |\mani|$ is a closed subset of $|\mani|$ we will say that $(\mathpzc{N}, \iota)$ is a closed embedded supermanifold.
\end{definition}
\noindent In what follows, we will always deal with closed embedded supermanifolds. 
Remarkably, it is possible to show that a morphism $\iota : \mathpzc{N} \rightarrow \mani$ is an embedding \emph{if and only if} the corresponding morphism $\iota^\sharp : \stsheaf \rightarrow \mathcal{O}_{\mathpzc{N}}$ is a surjective morphism of sheaves. Notice that, for example, given a supermanifold $\mani,$ one always has a natural closed embedding: the map $\iota : \manir \rightarrow \mani $, that embeds the reduced manifold underlying the supermanifold into the supermanifold itself.\\

\noindent We now introduce some further pieces of information carried by a supermanifold.
\begin{definition}[Nilpotent Sheaf / Fermionic Sheaf] We call the {nilpotent sheaf} $\mathcal{J}_\mani$ the sheaf of ideals of $\stsheaf = \mathcal{O}_{\mani, 0} \oplus \mathcal{O}_{\mani, 1}$ generated by all of the nilpotent sections, that is we put $\mathcal{J}_\mani \defeq \mathcal{O}_{\mani,1} \oplus \mathcal{O}_{\mani, 1}^2$. \\
Also, we call fermionic sheaf $\mathcal{F}_\mani$ the locally-free sheaf of $\mathcal{O}_{\manir}$-module of rank $0|m$ given by the quotient $\mathcal{F}_\mani \defeq \slantone{\mathcal{J}_\mani}{\mathcal{J}_{\mani}^2}$.
\end{definition}
\noindent It is crucial to note that modding out all of the nilpotent sections from the structure sheaf $\stsheaf $ of the supermanifold $\mani$ we recover the structure sheaf $\stsheafred$ of the  underlying ordinary complex manifold $\manir$, the local model was based on. We call the complex manifold $\manir$ the \emph{reduced manifold} of the supermanifold $\mani:$ loosely speaking, the reduced manifold arises by setting all of the nilpotents in $\stsheaf $ to zero.  \\
In other words, more invariantly, attached to any complex supermanifold there is a short exact sequence that relates the supermanifold with its reduced manifold:
\bear \label{ses}
\xymatrix@R=1.5pt{ 
0 \ar[rr] && \mathcal{J}_\mani \ar[rr] &&  \stsheaf  \ar[rr]^{\iota\; \; \;} && \stsheafred \ar[rr] && 0, 
}
\eear
where $\stsheafred \cong \slantone{\stsheaf}{\mathcal{J}_\mani}$ and the surjective sheaf morphism $\iota : \mathcal{O}_\mani \rightarrow \stsheafred$ corresponds to the existence of an \emph{embedding} $\manir \stackrel{\iota}{\longhookrightarrow} \mani$ of the reduced manifold $\manir $ inside the supermanifold $\mani$. Notice that $\mathcal{J}_\mani = \ker (\iota)$, where $\iota : \stsheaf \rightarrow \stsheafred$ is the surjective sheaf morphism in \eqref{ses}.\\ We will refer to the short exact sequence \eqref{ses} as the \emph{structural exact sequence} of $\mani$.

\noindent  A very natural question arising when looking at the structural exact sequence \eqref{ses} associated to a certain supermanifold is whether it is a \emph{split} exact sequence or not, that is whether there exists a {retraction} - called \emph{projection} in this context - $\pi : \stsheafred \rightarrow \stsheaf$ such that $\iota \circ \pi = id_{\stsheafred}:$
\bear \label{splittingseq}
\xymatrix@R=1.5pt{ 
0 \ar[rr] && \mathcal{J}_\mani \ar[rr] & &  \stsheaf  \ar[rr]_{\iota} && \ar@{-->}@/_1.3pc/[ll]_{\pi} \stsheafred \ar[rr] && 0.
 }
\eear
Notice that, more precisely, this shall be recasted into the splitting of two exact sequences - the even and the odd part of \eqref{ses} -, as we are only dealing with parity preserving morphisms. In particular, we shall give the following definition. 
\begin{definition}[Projected Supermanifold] We say that a supermanifold is projected if the \emph{even} part of its structural exact sequence \eqref{ses} splits:
\bear \label{project}
\xymatrix@R=1.5pt{ 
0 \ar[rr] && \mathcal{J}_{\mani,0} \ar[rr] & &  \mathcal{O}_{\mani, 0}  \ar[rr]_{\iota_0} && \ar@{->}@/_1.3pc/[ll]_{\pi_0} \stsheafred \ar[rr] && 0.
 }
\eear 
\end{definition}
\noindent It is important to observe that if the structure sheaf of a supermanifold is a sheaf of $\stsheafred$-modules if and only if the supermanifold is projected, indeed in this case one has that $\stsheaf \cong \stsheafred \oplus \mathcal{J}_\mani:$ is this case the theory simplifies considerably as the all of the sheaves of $\stsheaf$-modules defined on the supermanifold are also sheaves of $\stsheafred$-modules.\\
Notably, if also the \emph{odd} part of the structural exact sequence attached to the supermanifold $\mani$ is split, that is
\bear \label{oddses}
\xymatrix@R=1.5pt{ 
0 \ar[rr] && (\mathcal{J}_{\mani}^2)_1 \ar[rr] & &  \mathcal{O}_{\mani, 1}  \ar[rr]_{\iota_1} && \ar@{->}@/_1.3pc/[ll]_{\pi_1} \mathcal{F}_{\mani} \ar[rr] && 0,
 }
\eear 
then the supermanifold $\mani$ is called \emph{split}: this expresses in a more invariant and meaningful form the isomorphism $\mani \cong \mathfrak{S}(\manir, \Pi \mathcal{F}_\mani^\ast):$ the supermanifold is globally isomorphic to the local model it is based onto. In other words, we might say that a supermanifold $\mani$ is split if and only if is projected and the short exact sequence \eqref{oddses} is split. There indeed exists projected supermanifolds that are \emph{not} split. \\

Notice that all of the complex supermanifolds having odd dimension $1$ are projected and split because of dimensional reasons. When going up to odd dimension $2$ a supermanifold can instead be non-projected - the short exact sequence \eqref{project} tells that $\mathcal{O}_{\mani, 0}$ an extension of $\stsheafred$ by the line bundle $Sym^2 \mathcal{F}_\mani$. If we call $\mathcal{N}= 2$ supermanifold a complex supermanifolds having odd dimension equal to $2$, we have the following important result. 
\begin{theorem}[$\mathcal{N}=2$ Supermanifolds] \label{N=2} Let $\mani$ be a $\mathcal{N}=2$ supermanifold. Then $\mani$ is defined up to isomorphism by the triple 
$(\manir, \mathcal{F}_\mani, \omega_\mani)$ where
$\mathcal{F}_\mani$ is a rank $0|2$ sheaf of locally-free $\mathcal{O}_{\manir}$-modules, the fermionic sheaf of $\mani$, and $\omega_\mani \in H^1 (\manir, \mathcal{T}_{\manir} \otimes Sym^2 \mathcal{F}_\mani).$ The supermanifold $\mani$ is non-projected if and only if $\omega_\mani \neq 0$.
\end{theorem}
\noindent The proof of the statement can be originally found in \cite{Manin} and it has been reproduced in full details in \cite{NojaPhD}.\\

\section{Non-Projected $\mathcal{N}= 2$ Supermanifolds over $\proj 2$}

\noindent Using Theorem \ref{N=2} of the previous section, in the recent \cite{CNR} all the non-projected $\mathcal{N}=2$ supermanifolds over the projective plane $\proj 2$ were described through their characterizing cohomological invariants and their transition functions have been given. These non-projected supermanifolds reveal interesting features.

We first set out conventions: we consider a set of homogeneous coordinates $[X_0 : X_1 : X_2]$ on $\proj 2$ and the set of the affine coordinates and their algebras over the three open sets of the covering 
$\mathcal{U}\defeq \{ \mathcal{U}_0, \mathcal{U}_1, \mathcal{U}_2\}$ of $\proj 2$. In particular, modulo $\mathcal{J}^2_\mani$, we have the following  
\begin{align} \label{TransMod1}
& \mathcal{U}_0 \defeq \left \{ X_0 \neq 0 \right \} \quad \rightsquigarrow \quad z_{10} \,\mbox{mod}\,{\mathcal{J}^2_{\mathpzc{M}}}\defeq \frac{X_1}{X_0}, \qquad  z_{20} \,\mbox{mod}\,{\mathcal{J}^2_{\mathpzc{M}}} \defeq \frac{X_2}{X_0};   \nonumber \\
& \mathcal{U}_1 \defeq \left \{ X_1 \neq 0 \right \} \quad \rightsquigarrow \quad z_{11}\,\mbox{mod}\,{\mathcal{J}^2_{\mathpzc{M}}} \defeq \frac{X_0}{X_1}, \qquad z_{21} \,\mbox{mod}\,{\mathcal{J}^2_{\mathpzc{M}}}\defeq \frac{X_2}{X_1};   \nonumber \\  
& \mathcal{U}_2 \defeq \left \{ X_2 \neq 0 \right \} \quad \rightsquigarrow \quad z_{12} \,\mbox{mod}\,{\mathcal{J}^2_{\mathpzc{M}}} \defeq \frac{X_0}{X_2}, \qquad z_{22} \,\mbox{mod}\,{\mathcal{J}^2_{\mathpzc{M}}} \defeq \frac{X_1}{X_2}.  
\end{align}
The transition functions between these charts reads
\begin{align}\label{transfzeta2}
\mathcal{U}_{0} \cap \mathcal{U}_{1} : \qquad & z_{10} \,\mbox{mod}\,{\mathcal{J}^2_{\mathpzc{M}}}= \frac{1}{z_{11}}\,\mbox{mod}\,{\mathcal{J}^2_{\mathpzc{M}}},  \quad & z_{20} &\,\mbox{mod}\,{\mathcal{J}^2_{\mathpzc{M}}}= \frac{z_{21} }{z_{11} } \,
\mbox{mod}\,{\mathcal{J}^2_{\mathpzc{M}}}; \nonumber   \\
\mathcal{U}_{0} \cap \mathcal{U}_{2} : \qquad & z_{10} \,\mbox{mod}\,{\mathcal{J}^2_{\mathpzc{M}}} = \frac{z_{22}}{z_{12}} \,\mbox{mod}\,{\mathcal{J}^2_{\mathpzc{M}}}, \quad & z_{20} & \,\mbox{mod}\,{\mathcal{J}^2_{\mathpzc{M}}}= \frac{1}{z_{12}}\,
\mbox{mod}\,{\mathcal{J}^2_{\mathpzc{M}}}; \nonumber \\
\mathcal{U}_{1} \cap \mathcal{U}_{2} : \qquad & z_{11} \,\mbox{mod}\,{\mathcal{J}^2_{\mathpzc{M}}} = \frac{z_{12}}{z_{22}} \,\mbox{mod}\,{\mathcal{J}^2_{\mathpzc{M}}}, & \quad z_{21} & \,\mbox{mod}\,{\mathcal{J}^2_{\mathpzc{M}}} = \frac{1}{z_{22}} \,
\mbox{mod}\,{\mathcal{J}^2_{\mathpzc{M}}}.
\end{align}
Also, we denote $\theta_{1i},\ \theta_{2i}$ a basis of the rank $0|2$ locally-free sheaf $\mathcal{F}_\mani$ on any of the open sets $\mathcal{U}_i$, for $i=0,1,2$, and, since $\mathcal{J}_{\mani}^3=0$, the transition functions among these bases will have the form
\begin{align}\label{transtheta}
\mathcal{U}_{i} \cap \mathcal{U}_{j} : \qquad &  \left(\begin{array}{l} \theta_{1i}\\ \theta_{2i}\end{array}\right)=M_{ij}\cdot \left(\begin{array}{l} \theta_{1j}\\ \theta_{2j}\end{array}\right),
\end{align}
with $M_{ij}$ a $2\times 2$ matrix with coefficients in $\mathcal{O}_{\proj 2}(\mathcal{U}_i\cap \mathcal{U}_j)$. Note that in the transformation \eqref{transtheta} one can write $M_{ij}$ as a matrix with coefficients given by some even rational functions of 
$z_{1j},z_{2j}$, because of the definitions (\ref{TransMod1}) and the facts that $\theta_{hj}\in \mathcal{J}_\mani$ and $\mathcal{J}^3_\mani=0$.

Finally we note the transformation law for the products $\theta_{1i}\theta_{2i}$ is given by 
\bear 
\theta_{1i}\theta_{2i}=(\det M_{ij})\theta_{1j}\theta_{2j}. 
\eear 
Since $\det M$ is a transition function for the invertible sheaf $Sym^2 \mathcal{F}_\mani \cong \mathcal{O}_{\proj 2} (-3)$ over $\mathcal{U}_{i} \cap \mathcal{U}_{j}$, this can be written, up to constant changes of bases in $\mathcal{F}\lfloor_{\mathcal{U}_{i}}$ and $\mathcal{F}\lfloor_{\mathcal{U}_{j}}$, in the more precise form 
\begin{equation}\label{transtheta2}
\theta_{1i}\theta_{2i}=\left(\frac{X_j}{X_i}\right)^3\theta_{1j}\theta_{2j}.
\end{equation}
The meaning is that we can identify the base $\theta_{1i}\theta_{2i}$ of $Sym^2 \mathcal{F}_\mani \lfloor_{\mathcal{U}_i}$ with the standard base $\frac{1}{X_i^3}$ of $\mathcal{O}_{\proj 2} (-3)$ over $\mathcal{U}_{i}$.\\
Having set these conventions and notations we can give the following theorem, whose detailed proof can be found in \cite{CNR}.

\begin{theorem}[Non-Projected $\mathcal{N} =2 $ Supermanifolds over $\proj 2$] \label{pi2} Every non-projected $\mathcal{N} = 2$ supermanifold over $\proj 2$ is characterised up to isomorphism by a triple $\proj 2_\omega (\mathcal{F}_\mani) \defeq  (\proj 2, \mathcal{F}_\mani, \omega)$ where $\mathcal{F}_\mani$ is a rank $0|2$ sheaf of $\mathcal{O}_{\proj 2}$-modules such that $Sym^2 \mathcal{F}_\mani \cong \mathcal{O}_{\proj 2} (-3)$ and $\omega$ is a non-zero cohomology class $\omega \in H^1 (\mathcal{T}_{\proj 2} (-3)).$\\
 The transition functions for an element of the family $\proj 2_\omega (\mathcal{F}_\mani) $ from coordinates on $\mathcal{U}_0$ to coordinates on $\mathcal{U}_1$ are given by
\begin{align} \label{eq:eventranf}
\left ( 
\begin{array}{c}
z_{10} \\ 
z_{20} \\
\theta_{10} \\
\theta_{20} 
\end{array}
\right ) = \left ( \begin{array}{c@{}}  
 \dfrac{1}{z_{11}}  \\
 \dfrac{z_{21}}{z_{11}} + \lambda \dfrac{\theta_{11} \theta_{21}}{(z_{11})^2} \\
M  \left (
\begin{array}{c}
\theta_{11} \\
\theta_{21} 
\end{array}
\right ) 
\end{array}
\right ) 
\end{align}
where $\lambda \in \mathbb{C}$ is a representative of the class $\omega \in H^1 (\mathcal{T}_{\proj 2} (-3)) \cong \mathbb{C}$ and $M$ is a $2\times 2$ matrix with coefficients in $\mathbb{C}[z_{11}, z_{11}^{-1},z_{21}]$ such that $\det M=\slantone{1}{z_{11}^3}$. 
Similar transformations hold between the other pairs of open sets.
\end{theorem}
\noindent We remark that the form of transition functions above is shared by all the supermanifolds $\proj 2_\omega (\mathcal{F}_\mani)$, regardless the form of its fermionic sheaf $\mathcal{F}_\mani$, which is encoded in the matrix $M$. \\
Some remarkable properties of this family of non-projected supermanifolds has been given by the authors in \cite{CNR}. We condensate these results in the following theorem.
\begin{theorem}\label{embeddingth} Let $\mani $ be a non-projected supermanifold in the family $\proj {2}_\omega(\mathcal{F}_\mani).$ Then: \begin{enumerate}
\item $\mani$ is non-projective, that is $\mani$ cannot be embedded into any projective superspace of the kind $\proj {n|m}$;
\item $\mani$ can be embedded into a super Grassmannian. \\
In particular, let $\mathcal{T}_\mani$ be the tangent sheaf of $\mani$, if we let $V \defeq H^0(Sym^k\mathcal{T}_\mani)$, for any $k\gg 0$ the evaluation map $ev_\mani : V\otimes\mathcal{O}_\mani\to Sym^k\mathcal{T}_\mani$ induces an embedding \bear 
\Phi_k:\mani\rightarrow  G(2k|2k, V).\eear
\end{enumerate}
\end{theorem} 
\noindent We observe that the theorem proves the existence of an embedding into some super Grassmannian, but it is not effective in that it does not give an esteem of the symmetric power of the tangent sheaf needed in order to set up the embedding. In the next sections we will first review the geometry of super Grassmannians and then, we will treat explicitly an interesting example of embedding into a super Grassmannian, by choosing a decomposable fermionic sheaf satisfying the hypotheses of theorem \ref{pi2}.

\section{Elements of Super Grassmannians} \label{superG}

\noindent This section is dedicated to the introduction of some elements of geometry of super Grassmannians. We remark that the this section contains no original result and it is fully expository: all of the results are originally due to Y. Manin and his school, see in particular \cite{Manin}, \cite{ManinNC}, \cite{PenSko}. Nonetheless, we believe that since the cited literature is somewhat difficult and largely sketchy in the proofs of the various statements, it might be useful to have the constructions revised and readily at hand. In the present section our emphasis will be, anyway, on the non-projectedness and non-projectivity issues.

Super Grassmannians are the supergeometric generalisation of the ordinary Grassmannians. This means that $G(a|b; V^{n|m})$ is a universal parameter space for $a|b$-dimensional linear subspaces of a given $n|m$-dimensional space $V^{n|m}$. We will deal with the simplest possible situation, choosing the $n|m$-dimensional space $V^{n|m}$ to be a super vector space of the kind $\mathbb{C}^{n|m}$.\\
We start reviewing how to construct a super Grassmannian by patching together the \virgolette charts'' that cover it: this is a nothing but a generalization of the usual construction of ordinary Grassmannians making use of the so-called \emph{big cells}. 
\begin{enumerate}
\item We let $\mathbb{C}^{n|m}$ be such that $n|m = c_0|c_1 + d_0|d_1$ and look at $\mathbb{C}^{n|m}$ as given by $\mathbb{C}^{c_0 + d_0} \oplus (\Pi \mathbb{C})^{c_1 + d_1}$. This is obviously freely-generated, and we will write its elements as row 
vectors with respect to a certain basis, $\mathbb{C}^{n|m} = \mbox{Span} \{ e_1^{0}, \ldots, e_n^{0} | e^{1}_1, \ldots, e^{1}_m \},$ where the upper indices refer to the $\mathbb{Z}_2$-parity.
\item Consider a collection of indices $I = I_0 \cup I_1$ such that $I_0 $ is a collection of $d_0 $ out of the $n $ indices of $\mathbb{C}^n$ and $I_1$ is a collection of $d_1$ indices out of $m$ indices of $\Pi \mathbb{C}^m$. 
If $\mathcal{I}$ is the set of such collections of indices $I$ one gets 
\bear
\mbox{card}(\mathcal{I}) = \mbox{card} (\mathcal{I}_{0} \times \mathcal{I}_{1}) = {n \choose d_0 } \cdot {m \choose d_1 }.
\eear
This will give the number of \emph{super big cells} covering the super Grassmannian.
\item Choosing an element $I \in \mathcal{I}$, we associate to it a set of even and odd (complex) variables, we call them $\{ x^{\alpha \beta}_{I} \,| \, \xi^{\alpha \beta}_I \}$. These are arranged as to fill in the places of a $d_0|d_1 \times n|m = a|b \times (c_0 + d_0) | (c_1+d_1) $ matrix a way such that the columns having indices in $I \in \mathcal{I}_I$ forms a $(d_0 + d_1) \times (d_0 + d_1 )$ unit matrix if brought together. 
To makes this clear, for example, a certain choice of ${I} \in \mathcal{I}$ yields the following
\begin{equation}\label{matrixstform}
\mathcal{Z}_{{I}} \defeq 
\left (
\begin{array}{ccc|ccc||ccc|ccc}
& &  & 1 \; & & & & & & & & \\
\;  & \; x_I \; & \; &  & \ddots & & & 0 & &\;  &\; \xi_I \; &\;  \\
& & & & & \; 1& & & & & & \\
\hline \hline
& & & & & & 1\; & & & & & \\
&\; \xi_I \; & & & 0& & &\ddots & & &\;  x_I \; & \\
& & & & & & & &\; 1 & & & 
\end{array}
\right ),
\end{equation}
where we have chosen to pick that particular $ I \in \mathcal{I}$ that underlines the presence of the $(d_0 + d_1) \times (d_0 + d_1 )$ unit matrix.
\item We now define the superspace $\mathcal{U}_I \rightarrow \mbox{Spec}\,  \mathbb{C} \cong \{ pt\}$ to be the analytic superspace 
$\{ pt \} \times \mathbb{C}^{d_0 \cdot c_0 + d_1 \cdot c_1 | d_0\cdot c_1 + d_1\cdot c_0} \cong \mathbb{C}^{d_0 \cdot c_0 + d_1 \cdot c_1 | d_0\cdot c_1 + d_1\cdot c_0} $, where $\{ x^{\alpha \beta }_I \, | \, \xi^{\alpha \beta}_I \}$ are the complex 
coordinates over the point. Whenever is represented as above, the superspace related to $\mathcal{U}_I$ is called a \emph{super big cell} of the Grassmannian, and denoted with $\mathcal{Z}_I $ or, again, simply by $\mathcal{U}_I $ (which encodes the topological information).
\item We now show how to patch together two superspaces $\mathcal{U}_I$ and $\mathcal{U}_J$ for two different $I, J \in \mathcal{I}$. If $\mathcal{Z}_I$ is the super big cell related to $\mathcal{U}_I$, we consider the super submatrix $\mathcal{B}_{IJ}$ 
formed by the columns having indices in $J$. Let $\mathcal{U}_{IJ} = \mathcal{U}_I \cap \mathcal{U}_J$ be the (maximal) sub superspace of $\mathcal{U}_I$ such that on $\mathcal{U}_{IJ}$ the submatrix $\mathcal{B}_{IJ}$ is invertible. As usual, the odd 
coordinates do not affect the invertibility, so that it is enough that the two determinants of the even parts of the matrix $\mathcal{B}_{IJ}$ (that are respectively a $d_0 \times d_0 $ and a $d_1 \times d_1$ matrix) are different from zero. 
When this is the case, on the superspace $\mathcal{U}_{IJ}$ one has common coordinates $\{x^{\alpha \beta}_I \, | \, \xi^{\alpha \beta}_{I} \}$ and $\{ x^{\alpha \beta}_J \, |\, \xi^{\alpha \beta}_J \}$, and the rule to pass from one system of coordinates 
to the other one is provided by $\mathcal{Z}_J  = \mathcal{B}^{-1}_{IJ} \mathcal{Z}_I.$ 

\noindent For example, let us consider the following two super big cells:
\bear
\mathcal{Z}_I \defeq \left ( \begin{array}{ccc||cc}
1 & 0 & x_1 & 0 & \xi_1 \\
0 & 1 & x_2 & 0 & \xi_2 \\
\hline \hline 
0 & 0 & \eta & 1 & y
\end{array}
\right ),
\qquad \qquad 
\mathcal{Z}_J \defeq \left ( \begin{array}{ccc||cc}
1 & \tilde{x}_1 & 0 & 0 & \tilde{\xi}_1 \\
0 & \tilde{x}_2 & 1 & 0 & \tilde{\xi}_2 \\
\hline \hline 
0 & \tilde \eta & 0 & 1 & \tilde{y}
\end{array}
\right ).
\eear          
Looking at $\mathcal{Z}_I$, we see that the columns belonging to $J$ are the first, the third and the fourth, so that 
\bear
\mathcal{B}_{IJ} = \left (
\begin{array}{cc|c}
1 & x_1 & 0 \\
0 & x_2 & 0 \\
\hline
0 & \eta & 1
\end{array}
\right ).
\eear
Computing the determinant of the upper-right $2\times2$ matrix, we have invertibility of $\mathcal{B}_{IJ}$ corresponds to $x_2 \neq 0$ (as seen from the point of view of $\mathcal{U}_I$. Likewise one would have found $\tilde x_2 \neq 0$ by looking at 
$\mathcal{Z}_J$ and $\mathcal{U}_J$). The inverse of $\mathcal{B}^{-1}_{IJ}$ is 
\bear
\mathcal{B}^{-1}_{IJ} = \left (
\begin{array}{cc|c}
1 & -x_1/x_2 & 0 \\
0 & 1/x_1 & 0 \\
\hline
0 & \eta/x_2 & 1
\end{array}
\right )
\eear
so that we can compute the coordinates of $\mathcal{U}_J$ as functions of the ones of $\mathcal{U}_I$ via the rule $\mathcal{Z}_J = \mathcal{B}^{-1}_{IJ} \mathcal{Z}_I$:
\bear
\left ( \begin{array}{ccc||cc}
1 & \tilde{x}_1 & 0 & 0 & \tilde{\xi}_1 \\
0 & \tilde{x}_2 & 1 & 0 & \tilde{\xi}_2 \\
\hline \hline 
0 & \tilde \eta & 0 & 1 & \tilde{y}
\end{array}
\right )
= \left ( \begin{array}{ccc||cc}
1 & -x_1/x_2 & 0 & 0 & \xi_1 - \xi_2 x_1/x_2 \\
0 & 1/x_2 & 1 & 0 & \xi_2/x_2 \\
\hline \hline 
0 & -\eta/x_2 & 0 & 1 & y_1 - \eta \xi_2/ x_2
\end{array}
\right ),
\eear
so that the change of coordinates can be read out of this. Observe that the denominator $x_2 $ is indeed invertible on $\mathcal{U}_{IJ}.$
\item Patching together the superspaces $\mathcal{U}_I$ one obtains the Grassmannian supermanifold $G(d_0|d_1; \mathbb{C}^{n|m})$ as the quotient supermanifold 
\bear
G (d_0|d_1; \mathbb{C}^{n|m}) \defeq \slanttwo{\bigcup_{I \in \mathcal{I}}}{\mathcal{R}},
\eear 
where we have written $\mathcal{R}$ for the equivalence relations generated by the change of coordinates that have been described above. Notice that, as a (complex) supermanifold a super Grassmannian has dimension 
\bear
\dim_{\mathbb{C}} G(d_0 |d_1; \mathbb{C}^{n|m}) = d_0  (n-d_0) + d_1  (m-d_1) | d_0 (m-d_1) + d_1 (n-d_0). 
\eear  
We stress that the maps $\psi_{\mathcal{U}_I} : \mathcal{U}_I \rightarrow G(d_0|d_1; \mathbb{C}^{n|m})$ are isomorphisms onto (open) sub superspaces of the super Grassmannian, so that the various super big cells offer a \emph{local} description of it, 
in the same way a usual (complex) supermanifold is locally isomorphic to a superspace of the kind $\mathbb{C}^{n|m}.$  
\end{enumerate}
Clearly, the easiest possible example of super Grassmannians are projective superspaces, that are realised as $\proj {n|m} = G (1|0; \mathbb{C}^{n+1|m})$, exactly as in the ordinary case: these are \emph{split} supermanifolds, a feature that they do not in general 
share with a generic Grassmannian $G (d_0|d_1; \mathbb{C}^{n|m})$, as we shall see in a moment. \\
For convenience, in what follows we call $G$ a super Grassmannian of the kind $G(d_0|d_1; \mathbb{C}^{n|m})$ and we give the following, see \cite{Manin}.
\begin{definition}[Tautological Sheaf on a Super Grassmannian] Let $G$ be a super Grassmannian and let it be covered by the super big cells $\{ \mathcal{U}_I\}_{I \in \mathcal{I}}$. We call tautological sheaf $\mathcal{S}_G$ of the super Grassmannian $G$ the sheaf of locally-free 
$\mathcal{O}_G$-modules of rank $d_0|d_1$ defined as
\bear
\mathcal{U} \cap \mathcal{U}_I \longmapsto \mathcal{S}_G (\mathcal{U} \cap \mathcal{U}_I) \defeq \big \langle \mbox{\emph{rows of the matrix }} \mathcal{Z}_I \big \rangle_{\mathcal{O}_G(\mathcal{U} \cap \mathcal{U}_I)}.
\eear 
\end{definition}
\noindent Notice that this definition is well-posed, since one has that $\mathcal{S}_{G} (\mathcal{U}_I) \lfloor_{\mathcal{U}_{IJ}} $ and $\mathcal{S}_{G} ( \mathcal{U}_J ) \lfloor_{\mathcal{U}_{JI}}$ get identified by means of 
the transition functions $\mathcal{B}_{IJ}.$ \\
One can have insights about the geometry of a super Grassmannian by looking at its reduced space - which, we recall, encloses all the topological information -, and also at the filtration of its trivial sheaf $\mathcal{O}_G$. \\
We start observing that given a super Grassmannian $G$, one automatically has two ordinary \emph{even} sub Grassmannians.
\begin{definition}[$G_0$ and $G_1$] Let $G = G(d_0|d_1; \mathbb{C}^{n|m})$ be a super Grassmannian. Then we call $G_0$ and $G_1$ the two purely even sub Grassmannians defined as
\bear
G_0 \defeq G(d_0|0; \mathbb{C}^{n|0}), \qquad \qquad G_1 \defeq G (0|d_1; \mathbb{C}^{0|m}).
\eear 
\end{definition}
\noindent Given a super big cell $\mathcal{U}_I$, then $G_0$ and $G_1$ can be visualized as the upper-left and the lower-right part respectively and they come endowed with their tautological sheaves, we call them $\mathcal{S}_0 $ and $\mathcal{S}_1$. Notice, though, 
that $\mathcal{S}_1$ defines a sheaf of locally-free $\mathcal{O}_{G_1}$-modules and, as such, it has rank $0|d_1$.\\
Let us now consider an ordinary even complex Grassmannian $G$ of the kind $G (d; \mathbb{C}^n)$ together with its tautological sheaf $\mathcal{S}_G$. One can then also define the \emph{sheaf orthogonal to the tautological sheaf}, we call it $\widetilde{\mathcal{S}}$, whose dual fits into 
the short exact sequence
\bear \label{tautclassic}
\xymatrix@R=1.5pt{
0 \ar[rr] && \mathcal{S}_G \ar[rr] && \mathcal{O}_{G}^{\oplus n} \ar[rr] && \widetilde{\mathcal{S}}_G^\ast \ar[rr] &&  0.
}  
\eear
Notice that in the case the Grassmannian corresponds to a certain projective space $G(1|0; \mathbb{C}^{n+1}) = \proj n$, the sheaf orthogonal to the tautological sheaf can be red off the Euler exact sequence twisted by the tautological sheaf itself 
$\mathcal{S}_{\proj n} = \mathcal{O}_{\proj n} (-1)$, and, indeed, we have that $\widetilde{\mathcal{S}}_G^* \cong \mathcal{T}_{\proj n} (-1)$, so that $\widetilde{\mathcal{S}}_G \cong \Omega^1_{\proj n} (+1).$\\
In the case of a super Grassmannian $G(d_0|d_1; n|m)$ the sequence (\ref{tautclassic}) gets generalized to the canonical sequence
\bear \label{tautsuper}
\xymatrix@R=1.5pt{
0 \ar[rr] && \mathcal{S}_G \ar[rr] && \mathcal{O}_{G}^{\oplus n|m} \ar[rr] && \widetilde{\mathcal{S}}_G^\ast \ar[rr] &&  0.
}  
\eear
\noindent Recalling that $\mbox{Gr}\, \mathcal{O}_G \defeq \bigoplus_{i}^m \mbox{Gr}_i\,\mathcal{O}_G$ and $\mbox{Gr}_i \, \mathcal{O}_G \defeq \mathcal{J}_G^i /\mathcal{J}^{i+1}_G$, we now have all the ingredients to state the following theorem, whose proof is contained in \cite{Manin}.
\begin{theorem} \label{redg} Let $G = G(d_0|d_1; \mathbb{C}^{n|m})$ be a super Grassmannian and let $G_0$ and $G_1$ their even sub Grassmannians together with the sheaves $\mathcal{S}_0, \mathcal{S}_1$ and 
$\widetilde{\mathcal{S}}_0, \widetilde{\mathcal{S}}_1$. Then the following (canonical) isomorphisms hold true
\begin{itemize}
\item[1)] $G_{red} \cong  G_0 \times G_1$; 
\item[2)] $\mbox{\emph{Gr}}\, \mathcal{O}_G \cong Sym \,(\mathcal{S}_0 \otimes \widetilde{\mathcal{S}}_1 \oplus \widetilde{\mathcal{S}}_0 \otimes \mathcal{S}_1)$,
\end{itemize} 
where by $Sym$ we mean the super-symmetric algebra over $\mathcal{O}_{G_0 \times G_1}.$ 
\end{theorem}
\noindent The fundamental example, yet enclosing all the features characterizing the peculiar geometry of super Grassmannians, is given by $G(1|1, \mathbb{C}^{2|2})$ - which is of dimension $2|2$. We now study its geometry in some details. \vspace{6pt}\\
\noindent {\bf The Geometry of \boldmath{${G (1|1;{\pmb{{\mathbb{C}}}}^{2|2} )}$}:} we start studying the geometry of $G (1|1; \mathbb{C}^{2|2})$, we call it $G$ for short, from its reduced manifold which is easily identified using the 
previous Theorem \ref{redg}. 
\begin{lemma}[$G(1|1; \mathbb{C}^{2|2})_{red} \cong \proj 1_0 \times \proj 1_1$] Let $G$ be the super Grassmannian as above, then
\bear
G(1|1; \mathbb{C}^{2|2})_{red} \cong \proj 1_0 \times \proj 1_1.
\eear
\end{lemma}
\begin{proof} Keeping the same notation as above, one gets $G_0 = G(1|0; \mathbb{C}^{2|0})$ and $G_1 = G(0|1; \mathbb{C}^{0|2})$. Therefore, topologically, one has $G_0 \cong \proj 1_0$ and $G_1 \cong \proj 1_1$, where the subscripts refer to the two copies of projective lines. The conclusion follows by the first point of previous theorem. \end{proof}
\noindent It is fair to observe that we would have gotten to the same conclusion by looking at the big cells of this super Grassmannian, after having set the nilpotents to zero. \\ 
We thus have the following situation 
\bear
\xymatrix{
& \ar[dl]_{\pi_0} \proj 1_0 \times \proj 1_1 \ar[dr]^{\pi_1}\\
\proj 1_0  & & \proj 1_1}  
\eear
that helps us to recover the geometric data of $G_{red}$ and $G$ out of those of the two copies of projective lines. \\
Along this line, we recall that $\mathcal{O}_{\proj 1 \times \proj 1} (\ell_1 , \ell_2)$ is the \emph{external tensor product} 
$\mathcal{O}_{\proj 1_0} (\ell_1) \boxtimes \mathcal{O}_{\proj 1_1} (\ell_2) \defeq \pi^\ast_0 \mathcal{O}_{\proj 1_0}(\ell_1) \otimes_{\mathcal{O}_{\proj 1_0 \times \proj 1_1}} \pi^\ast_1 \mathcal{O}_{\proj 1_1} (\ell_2)$, and since the 
tautological sheaf on $\proj 1$ is $\mathcal{O}_{\proj 1} (-1),$ we have that 
\begin{align}
& \mathcal{S}_0 = \mathcal{O}_{\proj 1_0} (-1) \boxtimes \mathcal{O}_{\proj 1_1} = \mathcal{O}_{\proj 1_0 \times \proj 1_1} (-1,0), \\
& \mathcal{S}_1 = \Pi \mathcal{O}_{\proj 1_0} \boxtimes \mathcal{O}_{\proj 1_1} (-1) = \Pi \mathcal{O}_{\proj 1_0 \times \proj 1_1} (0,-1).
\end{align}
Similarly, observing that the sheaf dual to the tautological sheaf on $\proj 1$ is given again by the sheaf $\mathcal{O}_{\proj 1}(+1)$, as the (twisted) Euler sequence reads 
\bear
\xymatrix@R=1.5pt{
0 \ar[r] & \mathcal{O}_{\proj 1} (-1) \ar[r] & \mathcal{O}_{\proj 1}^{\oplus 2}  \ar[r] & \mathcal{T}_{\proj 1} (-1) \ar[r] &  0 
},  
\eear
and therefore $\widetilde{\mathcal{S}}_{\proj 1} \cong (\mathcal{T}_{\proj 1} (-1))^\ast \cong \Omega^1_{\proj 1} (+1) \cong \mathcal{O}_{\proj 1} (-1),$ one has the following:
\begin{align}
& \widetilde{\mathcal{S}}_0 = \mathcal{O}_{\proj 1_0} (-1) \boxtimes \mathcal{O}_{\proj 1_1} = \mathcal{O}_{\proj 1_0 \times \proj 1_1} (-1,0) , \\
& \widetilde{\mathcal{S}}_1 = \Pi \mathcal{O}_{\proj 1_0} \boxtimes \mathcal{O}_{\proj 1_1} (-1) = \Pi \mathcal{O}_{\proj 1_0 \times \proj 1_1} (0,-1).
\end{align}
This is enough to identify the fermionic sheaf of $G$, since $\mathcal{F}_G = \mbox{Gr}^{(1)} \,\mathcal{O}_G$ and therefore by virtue of the second point of the previous Theorem \ref{redg}, one has 
$\mathcal{F}_G \cong \mathcal{S}_0 \otimes \widetilde{\mathcal{S}}_1 \oplus \widetilde{\mathcal{S}}_0 \otimes \mathcal{S}_1 $, so
\bear
\mathcal{F}_G \cong \Pi \left ( \mathcal{O}_{\proj 1_0 \times \proj 1_1} (-1,-1) \oplus \mathcal{O}_{\proj 1_0 \times \proj 1_1} (-1,-1) \right ),
\eear
Which, in turns, shows that 
\bear
Sym^2 \mathcal{F}_G = \mathcal{O}_{\proj 1_0 \times \proj 1_1} (-2,-2).
\eear
and one can prove the following.
\begin{theorem}[$G(1|1; \mathbb{C}^{2|2})$ is Non-Projected] The supermanifold $G = G(1|1; \mathbb{C}^{2|2})$ is in general non-projected. In particular,
$H^1 (\mathcal{T}_{\proj 1_0 \times \proj 1_1} \otimes Sym^2 \mathcal{F}_G) \cong \mathbb{C}\oplus \mathbb{C}.$
\end{theorem}
\begin{proof} In order to compute the cohomology group $H^1(\mathcal{T}_{\proj 1_0 \times \proj 1_1} \otimes Sym^2 \mathcal{F}_G)$, we observe that in general, on the product of two varieties, we have 
$\mathcal{T}_{X \times Y} \cong p_1^\ast \mathcal{T}_X \oplus p_2^\ast \mathcal{T}_{Y}$, where the $p_i$ are the projections on the factors, so that, in particular, we find 
\bear
\mathcal{T}_{\proj 1_0 \times \proj 1_1 } \cong \pi_0^\ast \mathcal{T}_{\proj 1_0} \oplus \pi_1^\ast \mathcal{T}_{\proj 1_1} \cong \pi^\ast_0 \mathcal{O}_{\proj 1_0} (2) \oplus \pi^\ast_1 \mathcal{O}_{\proj 1_1} (2) = 
\mathcal{O}_{\proj 1_0 \times \proj 1_1} (2,0) \oplus \mathcal{O}_{\proj 1_0 \times \proj 1_1} (0,2).  \nonumber
\eear 
Taking the tensor product with $Sym^2 \mathcal{F}_G$, one has
\begin{align}
\mathcal{T}_{\proj 1_0 \times \proj 1_1 } \otimes Sym^2 \mathcal{F}_G & \cong \left ( \mathcal{O}_{\proj 1_0 \times \proj 1_1} (2,0) \oplus \mathcal{O}_{\proj 1_0 \times \proj 1_1} (0,2) \right ) \otimes \mathcal{O}_{\proj 1_0 \times \proj 1_1} (-2,-2) \nonumber \\
 & \cong \mathcal{O}_{\proj 1_0 \times \proj 1_1} (0,-2) \oplus \mathcal{O}_{\proj 1_0 \times \proj 1_1} (-2,0). 
\end{align}
Now, via the K\"unneth formula one has 
\bear
H^n \left ( X \times Y , p_1^\ast \mathcal{F}_{X} \otimes_{\mathcal{O}_{X \times Y}} p^\ast_2 \mathcal{G}_Y \right ) \cong \bigoplus_{i+j = n} H^i \left ( X, \mathcal{F}_X \right ) \otimes H^j (Y, \mathcal{F}_Y), 
\eear 
so that 
\begin{align}
H^1 (\mathcal{T}_{\proj 1_0 \times \proj 1_1 } \otimes Sym^2 \mathcal{F}_G )& \cong H^1 ( \mathcal{O}_{\proj 1_0 \times \proj 1_1} (0,-2) \oplus \mathcal{O}_{\proj 1_0 \times \proj 1_1} (-2,0)) \nonumber \\
& \cong H^1 ( \mathcal{O}_{\proj 1_0 \times \proj 1_1} (0,-2)) \oplus H^1 (\mathcal{O}_{\proj 1_0 \times \proj 1_1} (-2,0))  \nonumber \\
& \cong H^0 ( \mathcal{O}_{\proj 1_0 }) \otimes H^1 (\mathcal{O}_{\proj 1_1}) (-2) \oplus H^1 ( \mathcal{O}_{\proj 1_0 }) (-2) \otimes H^0 (\mathcal{O}_{\proj 1_1}) \nonumber \\
& \cong \mathbb{C} \oplus \mathbb{C},
\end{align}
which concludes the proof.
\end{proof}
\noindent There are different ways to find the representatives in the obstruction cohomology group for $G$. We will first use the super big cells of $G(1|1; \mathbb{C}^{2|2})$ to identifies these representatives and to establish that in the isomorphisms $H^1(\mathcal{T}_{\proj 1_0 \times \proj 1_1 } \otimes Sym^2 \mathcal{F}_G) \cong \mathbb{C} \oplus \mathbb{C}$ the cohomology class corresponds to the choice $\omega_G = (1 , 1)$. This is an explicit and immediate way to do this. \\
First, we observe that, since the reduced manifold underlying $G(1|1; \mathbb{C}^{2|2})$ has the topology of $\proj 1_0 \times \proj 1_1$, it is covered by four open sets. If we call 
$\mathcal{U}^{(0)} = \{ \mathcal{U}^{(0)}_\ell \}_{\ell = 0,1}$ the usual open sets covering $\proj 1_0 $ and $\mathcal{U}^{(1)} = \{ \mathcal{U}^{(1)}_\ell \}_{\ell = 0,1}$ the open sets covering $\proj 1_1$, we then have a system of open sets 
covering their product $\proj 1_0 \times \proj 1_1$ given by 
\begin{align}
& \mathcal{U}_1 \defeq \mathcal{U}^{(0)}_0 \times \mathcal{U}^{(1)}_0 = \left \{ ([X_0: X_1], [Y_0 : Y_1] ) \in \proj 1_0 \times \proj 1_1 : X_0 \neq 0 , \; Y_0 \neq 0 \right \},  \cr
& \mathcal{U}_2 \defeq \mathcal{U}^{(0)}_1 \times \mathcal{U}^{(1)}_0 = \left \{ ([X_0: X_1], [Y_0 : Y_1] ) \in \proj 1_0 \times \proj 1_1: X_1 \neq 0 , \; Y_0 \neq 0 \right \}, \cr
& \mathcal{U}_3 \defeq \mathcal{U}^{(0)}_0 \times \mathcal{U}^{(1)}_1 = \left \{ ([X_0: X_1], [Y_0 : Y_1] ) \in \proj 1_0 \times \proj 1_1 : X_0 \neq 0 , \; Y_1 \neq 0 \right \}, \cr
& \mathcal{U}_3 \defeq \mathcal{U}^{(0)}_1 \times \mathcal{U}^{(1)}_1 = \left \{ ([X_0: X_1], [Y_0 : Y_1] ) \in \proj 1_0 \times \proj 1_1 : X_1 \neq 0 , \; Y_1 \neq 0 \right \}. 
\end{align}
These correspond to the following matrices $\mathcal{Z}_{\mathcal{U}_i}$, out of which we can read the coordinates on the big cells:
\begin{align}
\mathcal{Z}_{\mathcal{U}_1} \defeq \left ( \begin{array}{ccc||ccc} 
1 & & x_1 & 0 & &  \xi_1 \\
\hline \hline
0 & & \eta_1  & 1 & & y_1
\end{array}
\right ), \qquad 
\mathcal{Z}_{\mathcal{U}_2} \defeq \left ( \begin{array}{ccc||ccc} 
x_2 & & 1 & 0 & &  \xi_2 \\
\hline \hline
\eta_2 & & 0  & 1 & & y_2
\end{array}
\right ), \\ \nonumber \\
\mathcal{Z}_{\mathcal{U}_3} \defeq \left ( \begin{array}{ccc||ccc} 
1 & & x_3 & \xi_3 & &  0 \\
\hline \hline
0 & & \eta_3  & y_3 & & 1
\end{array}
\right ), \qquad 
\mathcal{Z}_{\mathcal{U}_4} \defeq \left ( \begin{array}{ccc||ccc} 
x_4 & & 1 & \xi_4 & &  0 \\
\hline \hline
\eta_4 & & 0  & y_4 & & 1
\end{array}
\right ).
\end{align}
Following the procedure illustrated above or by rows and columns operations on the $\mathcal{Z}_{\mathcal{U}_i}$ one find the transition rules between the various charts, 
\begin{align}
& \mathcal{U}_1 \cap \mathcal{U}_2 \; \rightsquigarrow \; \left \{ \begin{array}{l}
x_1 = x_2^{-1} \\
\xi_1 = \xi_2 x_2^{-1}  \\
\eta_1 = - \eta_2 x_2^{-1}  \\
y_1 = y_2 + \xi_2 \eta_2 x_2^{-1}
\end{array}
\right.
\qquad \mathcal{U}_1 \cap \mathcal{U}_3 \; \rightsquigarrow \; \left \{ \begin{array}{l}
x_1 = x_3 - \xi_3 \eta_3 y^{-1}_3 \\
\xi_1 = - \xi_3 y_3^{-1}  \\
\eta_1 = \eta_3 y_3^{-1}  \\
y_1 = y_3^{-1}
\end{array}
\right. \nonumber \\ \nonumber \\
& \mathcal{U}_1 \cap \mathcal{U}_4 \; \rightsquigarrow \; \left \{ \begin{array}{l}
x_1 = x_4^{-1} + \xi_4 \eta_4 x_4^{-2} y_4^{-1}\\
\xi_1 = - \xi_4 x_4^{-1} y_4^{-1} \\
\eta_1 = - \eta_4 x_4^{-1} y_4^{-1} \\
y_1 = y_4^{-1} - \xi_4 \eta_4 x_4^{-1} y^{-2}_4
\end{array}
\right. 
\qquad \mathcal{U}_2 \cap \mathcal{U}_3 \; \rightsquigarrow \; \left \{ \begin{array}{l}
x_2 = x_3^{-1} + \xi_3 \eta_3 x_3^{-2} y_3^{-1}\\
\xi_2 = - \xi_3 x_3^{-1} y_3^{-1} \\
\eta_2 = - \eta_3 x_3^{-1} y_3^{-1} \\
y_2 = y_3^{-1} - \xi_3 \eta_3 x_3^{-1} y^{-2}_3
\end{array}
\right. \nonumber \\ \nonumber \\
& \mathcal{U}_2 \cap \mathcal{U}_4 \; \rightsquigarrow \; \left \{ \begin{array}{l}
x_2 = x_4 - \xi_4 \eta_4 y^{-1}_4 \\
\xi_2 = - \xi_4 y_4^{-1} \\
\eta_2 = \eta_4 y_4^{-1} \\
y_2 = y_4^{-1}
\end{array}
\right. \qquad
\mathcal{U}_3 \cap \mathcal{U}_4 \; \rightsquigarrow \; \left \{ \begin{array}{l}
x_3 = x_4^{-1} \\
\xi_3 = \xi_4 x_4^{-1} \\
\eta_3 = - \eta_4 x_4^{-1} \\
y_3 = y_4 + \xi_4 \eta_4 x_4^{-1}
\end{array}
\right.
\end{align}
By looking at these transformation rules, we therefore have that in the isomorphism above the class is represented by $(1,1) \in \mathbb{C} \oplus \mathbb{C}$ and the cocycles representing $\omega$ are given by 
$\omega = (\omega_{12}, \omega_{13}, \omega_{14}, \omega_{23}, \omega_{24}, \omega_{34} ),$ where the $\omega_{ij}$ are (in tensor notation)
\begin{align}
&\omega_{12} = \frac{\xi_2 \eta_2}{x_2} \otimes \partial_{y_1} , \qquad  \qquad \qquad \qquad \quad \omega_{13} =  - \frac{\xi_3 \eta_3}{y_3} \otimes \partial_{x_1},  \nonumber \\
& \omega_{14} = + \frac{\xi_4 \eta_4}{x_4^2 y_4} \otimes \partial_{x_1}  - \frac{\xi_4 \eta_4}{x_4 y^2_4} \otimes \partial_{y_1}, \qquad \omega_{23} = + \frac{\xi_3 \eta_3}{x_3^2 y_3} \otimes \partial_{x_2} - \frac{\xi_3 \eta_3}{x_3 y^2_3} \otimes \partial_{y_2}, 
\nonumber \\
&\omega_{24} =  - \frac{\xi_4 \eta_4}{y_4} \otimes \partial_{x_2}, \qquad \qquad \qquad \qquad \; \omega_{34} = + \frac{\xi_1 \eta_4}{x_4} \otimes \partial_{y_3}.
\end{align}
One can get to the same result also by means of a different computation, as remarked above. Observing that $H^1 (\mathcal{O}_{\proj 1_0 \times \proj 1_1} (-2,0)) \oplus H^1 (\mathcal{O}_{\proj 1_0 \times \proj 1_1} (0,-2))$ is generated by the two elements 
\bear
H^1 (\mathcal{O}_{\proj 1_0 \times \proj 1_1} (-2,0)) \oplus H^1 (\mathcal{O}_{\proj 1_0 \times \proj 1_1} (0,-2)) \cong \left \langle \frac{1}{X_0 X_1} \boxtimes 1 ,\; 1 \boxtimes \frac{1}{Y_0 Y_1} \right \rangle_{\mathcal{O}_{\proj 1_0 \otimes \proj 1_1}}, 
\eear
we can then look at these generators in the intersections, keeping in mind that $\mathcal{F}_G \cong \Pi \mathcal{O}_{\proj 1_0 \times \proj 1_1} (-1,-1) 
\oplus \Pi \mathcal{O}_{\proj 1_0 \times \proj 1_1} (-1,-1),$ in order to identify the cocycles that enter in the transition functions. We examine the various intersections.
\begin{itemize}
\item[ $\mathcal{U}_1 \cap \mathcal{U}_2:$] The following identifications can be made
\begin{align}
& \xi_1 = \Pi \left (\frac{1}{X_0} \boxtimes \frac{1}{Y_0} , 0 \right ), \qquad \eta_1 = \Pi \left (0, \frac{1}{X_0} \boxtimes \frac{1}{Y_0}  \right ),  \cr
& \xi_2 = \Pi \left (\frac{1}{X_1} \boxtimes \frac{1}{Y_0} , 0 \right ), \qquad \eta_2 = \Pi \left (0, \frac{1}{X_1} \boxtimes \frac{1}{Y_0}  \right ).  
\end{align}
These gives the transition functions above between $\xi_1 $ and $\xi_2$ and between $\eta_1 $ and $\eta_2$. Notice that in the intersection $\mathcal{U}_1 \cap \mathcal{U}_2$ only the bit $H^1 (\mathcal{O}_{\proj 1_0 \times \proj 1_1} (-2,0))$ contributes and we have therefore
\begin{align}
\omega_{12} & = \pm \ell_1\left ( \frac{1}{X_0X_1} \boxtimes 1 \right ) = \pm \ell_1 \left ( \frac{1}{X_0 X_1} \boxtimes \frac{Y_0^2}{Y_0^2} \right ) = \pm \ell_1 \left ( \frac{1}{X_0 X_1}\boxtimes \frac{1}{Y_0^2} \right ) \otimes \partial_{y_1}  \nonumber \\
& = \pm \ell_1 \left ( \frac{X_1}{X_0} \right ) \left ( \Pi \left (\frac{1}{X_1} \boxtimes \frac{1}{Y_0} , 0 \right ) \odot \Pi \left (0, \frac{1}{X_1} \boxtimes \frac{1}{Y_0}\right ) \right ) \otimes \partial_{y_1} \nonumber  \\
& = \pm \ell_1 \frac{\xi_2 \eta_2 }{x_2} \otimes \partial_{y_1}
\end{align}
where we have denoted by $\odot$ the supersymmetric product of the two (local) sections on $\mathcal{F}_G,$ as represented above.
\item[$\mathcal{U}_1 \cap \mathcal{U}_3:$] here we have a contribution from $H^1  (\mathcal{O}_{\proj 1_0 \times \proj 1_1} (0,-2))$ only and, therefore, we have to deal with $\omega_{13} = \ell_2 \left (1 \boxtimes 1 / Y_0 Y_1 \right )$. By a completely 
analogous treatment as above, one finds that
\bear
\omega_{13} = \pm \ell_2 \left ( 1 \boxtimes \frac{1}{Y_0Y_1} \right ) = \pm \ell_2 \frac{\xi_3 \eta_3}{y_3} \otimes \partial_{x_1}.   
\eear
\item[$\mathcal{U}_1 \cap \mathcal{U}_4:$] In this case we have both the contributions, so
\begin{align}
\omega_{14} = \pm \ell_1 \left ( \frac{1}{X_0X_1} \boxtimes 1\right ) \pm \ell_2 \left ( 1 \boxtimes \frac{1}{Y_0 Y_1 }\right ), 
\end{align}
so that by analogous manipulations as the one above one finds
\bear
\omega_{14} = \pm \ell_1 \frac{\xi_4 \eta_4}{x_4 y_4^2} \otimes \partial_{y_1} \pm \ell_2 \frac{\xi_4 \eta_4}{x_4^2 y_4} \otimes \partial_{x_1}. 
\eear
\end{itemize}
All the other $\omega_{ij}$ are identified in the same way and enter one of these three categories. \\
To conclude, one then impose the cocycle conditions as to fix the various signs of the $\ell_1 $ and $\ell_2 $ above, that agree with the one we found above by looking at the coordinates of the big cells: choosing $(\ell_1 = 1 , \ell_2 = 1)$ - this can always be done up to a change of coordinates -, one gets the same even 
transition functions as above. \\
This is enough to use the theorem classifying the complex supermanifold of dimension $n|2$ (see \cite{Manin} or \cite{CNR}) as to conclude that $G(1|1; \mathbb{C}^{2|2})$ can be defined up to isomorphism as follows
\begin{definition}[$G(1|1; \mathbb{C}^{2|2})$ as a Non-Projected Supermanifold] The super Grassmnannian $G(1|1;\mathbb{C}^{2|2})$ can be defined up to isomorphism as the $2|2$ dimensional supermanifold characterised by the triple 
$(\proj 1_0 \times \proj 1_1, \mathcal{F}_G, \omega_G)$ where $\mathcal{F}_G = \Pi  \mathcal{O}_{\proj 1_0 \times \proj 1_1} (-1,-1) \oplus \Pi \mathcal{O}_{\proj 1_0 \times \proj 1_1} (-1,-1) $ and where $\omega_{G} = (\ell_1, \ell_2)$, with $\ell_1\neq 0 $ and $\ell_2 \neq 0$, in the isomorphism 
$\omega_G \in H^1(\mathcal{T}_{\proj 1_0 \times \proj 1_1} \otimes Sym^2 \mathcal{F}_G) \cong \mathbb{C} \oplus \mathbb{C}.$
\end{definition}
\noindent On a very general ground, apart from projective superspaces, super Grassmannians are in general non-projected: the case of $G(1|1; \mathbb{C}^{2|2})$ we treated is just
the first non-trivial example of non-projected super Grassmannian. \\
Now, jump to the second issue we are interested into: we show that $G(1|1; \mathbb{C}^{2|2})$ is \emph{not} a non-projective supermanifold.
\begin{theorem}[$G(1|1; \mathbb{C}^{2|2})$ is Non-Projective] Let $G(1|1; \mathbb{C}^{2|2})$ be super Grassmannian defined as above. Then $G(1|1; \mathbb{C}^{2|2})$ is non-projective.
\end{theorem}
\begin{proof} In order to prove the non-projectivity of $G \defeq G(1|1; \mathbb{C}^{2|2})$ we consider the following short exact sequence that comes from the structural exact sequence of $G$:
\bear \label{starshort}
\xymatrix@R=1.5pt{
0 \ar[r] & \mathcal{O}_{\proj 1_0 \times \proj 1_1} (-2, -2) \ar[r] & \mathcal{O}_{G,0}^*  \ar[r] & \mathcal{O}_{\proj 1_0 \times \proj 1_1}^\ast \ar[r] &  0. 
}  
\eear
Ordinary results in algebraic geometry yield $H^0 (\mathcal{O}_{\proj 1_0 \times \proj 1_1} (-2, -2)) = 0 = H^1 (\mathcal{O}_{\proj 1_0 \times \proj 1_1} (-2, -2) )$, whereas 
$H^2 (\mathcal{O}_{\proj 1_0 \times \proj 1_1} (-2, -2)) \cong \mathbb{C}.$ Likewise, one has $H^0 ( \mathcal{O}_{\proj 1_0 \times \proj 1_1}^\ast ) \cong \mathbb{C}^\ast$ and 
$\mbox{Pic} (\proj 1_0 \times \proj 1_1 ) = H^1 ( \mathcal{O}_{\proj 1_0 \times \proj 1_1}^\ast ) \cong \mathbb{Z} \oplus \mathbb{Z},$ by means of the ordinary exponential exact sequence. This is enough to realize that the cohomology sequence induced by the sequence above splits in two exact sequences. The first one gives an isomorphism $H^0 (\mathcal{O}_{G, 0}) \cong \mathbb{C}^\ast$, while the second one instead reads 
\bear
\!\!\!\! \xymatrix@R=1.5pt{
0 \ar[r] & H^1 (\mathcal{O}^\ast_{G, 0}) \ar[r] & \mbox{Pic} (\proj 1_0 \times  \proj 1_1 ) \cong \mathbb{Z}\oplus \mathbb{Z} \ar[r] & H^2 (\mathcal{O}_{\proj 1_0 \times \proj 1_1} (-2, -2))\cong \mathbb{C} \ar[r] &  \cdots. 
}  
\eear
Thus in order to establish the fate of the cohomology group $H^1 (\mathcal{O}_{G,0}^\ast)$ one has to look at the boundary map $\delta : \mbox{Pic} (\proj 1_0 \times \proj 1_1) \rightarrow H^2 (\mathcal{O}_{\proj 1_0 \times \proj 1_1} (-2, -2))$. Let then us consider the following diagram of cochain complexes:
\bear
\xymatrix{
{C}^2 (\mathcal{O}_{\proj 1_0 \times \proj 1_1} (-2, -2) ) \; \ar@{>->}[r] & {C}^2 (\mathcal{O}^\ast_{G, 0}) &  \\ 
& {C}^1 (\mathcal{O}^\ast_{G, 0}) \ar@{|->}[u] \ar@{|->>}[r] & {C}^1 (\mathcal{O}^\ast_{\proj 1_0 \times \proj 1_1}), 
}  
\eear
obtained by combining \eqref{starshort} with the \v{C}ech cochain complexes of the sheaves that appear.\\
Since $\langle \mathcal{O}_{\proj 1_0 \times \proj 1_1} (1, 0) , \mathcal{O}_{\proj 1_0 \times \proj 1_1} (0,1) \rangle_{\mathcal{O}_{\proj 1_0 \times \proj 1_1}}  \cong \mbox{Pic} (\proj 1_0 \times \proj 1_1)$, given the usual cover 
of $\proj 1_0 \times \proj 1_1$ by the open sets $\mathcal{U}_{i} $ above, $\mathcal{O}^\ast_{\proj 1_0 \times \proj 1_1} (1,0)$ could be represented by (six) cocycles 
$g_{ij} \in Z^1 (\mathcal{U}_i \cap \mathcal{U}_j, \mathcal{O}^\ast_{\proj 1_0 \times \proj 1_1})$. Explicitly, these cocycles are the transition functions of the line bundle 
\begin{align}
\mathcal{O}^\ast_{\proj 1_0 \times \proj 1_1} (1,0) \longleftrightarrow \left \{ g_{12} = \frac{X_1}{X_0}, g_{13} = 1, g_{14} = \frac{X_1}{X_0}, g_{23} = \frac{X_0}{X_1}, g_{24}= 1, g_{34} = \frac{X_1}{X_0} \right \},  \nonumber
\end{align}
where, with an abuse of notation, we dimissed the second bit of the external tensor product, which is just the identity. Since the map $j : {C}^1 (\mathcal{O}^\ast_{G, 0}) \rightarrow {C}^1 (\mathcal{O}^\ast_{\proj 1_0 \times \proj 1_1}) $ 
is surjective, these cocycles are images of elements in ${C}^1 (\mathcal{O}^\ast_{G, 0})$. Notice that $j$ is induced by the inclusion of the reduced variety $\proj 1_0 \times \proj 1_1$ into $G$, so the cochains in 
${C}^1 (\mathcal{O}^\ast_{G, 0})$ are exactly the $\{g_{ij}\}_{i,j\in I}$ we have written above (notice also that these are no longer cocycles in $\mathcal{O}^\ast_{G, 0}$). Using the \v{C}ech coboundary map 
$\delta (j^\ast \mathcal{O}_{\proj 1_0 \times \proj 1_1} (1,0))$ over $G$, one finds, for example:  
\begin{align}
g_{12} \cdot g_{23} \cdot g_{31} \lfloor_{\mathcal{U}_1 \cap \mathcal{U}_2 \cap \mathcal{U}_3} = 1 \boxtimes 1 + \frac{1}{X_0X_1} \boxtimes \frac{1}{Y_0 Y_1}.   
\end{align}
Indeed, by looking at the affine coordinates in the big cells, these reads $x_2 x_3 = 1 + \frac{\xi_2 \eta_2}{x_2 y_2}$ and setting, as we have done above above
\bear
\xi_2 = \Pi \left ( \frac{1}{X_1} \boxtimes \frac{1}{Y_0}, 0 \right ), \qquad \eta_2 = \Pi \left (0,  \frac{1}{X_1} \boxtimes \frac{1}{Y_0} \right ),  
\eear  
and taking their supersymmetric product one has $\frac{\xi_2 \eta_2}{x_2 y_2} = \frac{1}{X_0X_1} \boxtimes \frac{1}{Y_0Y_1} $. Now, by exactness of the diagram, this element is in the kernel of the map 
$j : {C}^2 (\mathcal{O}^\ast_{G, 0}) \rightarrow {C}^2 (\mathcal{O}^\ast_{\proj 1_0 \times \proj 1_1})$, that equals the image of the map 
$i : {C}^2 (\mathcal{O}_{\proj 1_0 \times \proj 1_1} (-2, -2) ) \rightarrow {C}^2 (\mathcal{O}^\ast_{G, 0})$, therefore there exists an element $N \in {C}^2 (\mathcal{O}_{\proj 1_0 \times \proj 1_1} (-2, -2) )$ such that 
$i(N) = 1 \boxtimes 1 + \frac{1}{X_0X_1} \boxtimes \frac{1}{Y_0 Y_1}$ and it is a cocycle. Then, considering that the map $i$ is induced by the map 
$ \mathcal{O}_{\proj 1_0 \times \proj 1_1} (-2,-2)  \owns a\boxtimes b \mapsto 1 \boxtimes 1 + a\boxtimes b \in \mathcal{O}^\ast_{G,0}$, we have that the element $1\boxtimes + \frac{1}{X_0X_1} \boxtimes \frac{1}{Y_0 Y_1}$ is the image of 
$1 \frac{1}{X_0X_1} \boxtimes \frac{1}{Y_0 Y_1}$ via $i$. By symmetry, the same applies to the second generator of $\mbox{Pic} (\proj 1_0 \times \proj 1_1)$, which is given by $\mathcal{O}_{\proj 1_0 \times \proj 1_1} (0, 1)$, thus that the map 
$ \delta : \mbox{Pic} (\proj 1_0 \times \proj 1_1) \cong \mathbb{Z} \oplus \mathbb{Z} \rightarrow H^2 (\mathcal{O}_{\proj 1_0 \times \proj 1_1} (-2, -2))  \cong \mathbb{C}$ reads $ \mathbb{Z}\oplus \mathbb{Z} \owns (a, b) \longmapsto a + b \in \mathbb{C}.$ 
By exactness, it follows that the only invertible sheaves on $\proj 1_0 \times \proj 1_1$ that lift to the whole $G$ are those of the kind $\mathcal{O}_{\proj 1_0 \times \proj 1_1}(a, -a)$, as the composition of the maps yields $(a,-a) \mapsto (a,-a) \mapsto 
a - a = 0 $ as it should. Since these invertible sheaves have no cohomology, they cannot give any embedding in projective superspaces and this completes the proof. \end{proof}
\noindent Notice the subtlety: the above theorem says that $\mbox{Pic} (\proj 1_0 \times \proj 1_1) \neq 0$ (actually $\mbox{Pic} (\proj 1_0 \times \proj 1_1) \cong \mathbb{Z}$), but still there are no \emph{ample} invertible sheaves that allow for an embedding of $G (1|1; \mathbb{C}^{2|2}) $ into some projective superspaces.\\
The fundamental consequence is that non-projectivity is not confined to this particular super Grassmannian only.
\begin{theorem}[Super Grassmannians are Non-Projective] The super Grassmannian space $G(a|b; \mathbb{C}^{m|n})$ for $0 < a < n$ and $0 < b < m$ is non-projective.
\end{theorem}
\begin{proof} as in \cite{Manin}, it is enough to observe that the inclusion $\mathbb{C}^{2|2} \subset \mathbb{C}^{a+1 | b+1}$ induces in turn the inclusion $G(1|1; \mathbb{C}^{2|2}) \hookrightarrow G (1|1; \mathbb{C}^{a+1 | b+1})$. This last super 
Grassmannian is isomorphic, as for the usual Grassmannians, to $G(a|b; (\mathbb{C}^{a+1|b+1})^\ast)$, that in turn embeds into $G(a|b; \mathbb{C}^{n|m})$. This leads to $G(1|1; \mathbb{C}^{2|2}) \hookrightarrow G(a|b; \mathbb{C}^{n|m}):$ as 
$G(1|1; \mathbb{C}^{2|2})$ is non-projective, so is $G(a|b; \mathbb{C}^{n|m})$, completing the proof.\end{proof}
\noindent The upshot of this result is that, working in the context of algebraic supergeometry, it is no longer true that projective superspaces are a privileged ambient: this is a substantial departure from usual context of complex algebraic geometry, that deserves to be stressed out.

\section{Maps and Embeddings into a Super Grassmannian: An Explicit Example}

\noindent Having reviewed the geometry of super Grassmannians in the previous section, we now consider the problem of setting up \emph{maps} to super Grassmannians. \\
First we recall the universal property characterizing the construction of maps into projective superspaces $\proj {n|m}$, which is nothing but a direct generalization of the usual criterium in algebraic geometry for projective spaces $\proj n$, using invertible sheaves, \emph{i.e.}\ for any supermanifold or superscheme $\mani$, any locally-free sheaf $\mathcal{L}$ of rank $1|0$ on $\mani$ and any vector superspace $V$ having a surjective sheaf-theoretical map $V\otimes \mathcal{O}_\mani \rightarrow \mathcal{L}$, then there exists a unique (up to isomorphisms) map $\Phi_{\mathcal{L}} : \mani \rightarrow \proj {n|m}$ such that the inclusion $\mathcal{L}^\ast \rightarrow V^\ast \otimes \mathcal{O}_\mani$ is the pull-back of the inclusion $\mathcal{O}_{\proj {n|m}} (-1) \to \mathcal{O}_{\proj {n|m}}^{\oplus n+1|m}$ coming from the Euler exact sequence. More concretely, this is sometimes reported simply asking $\mathcal{L}$ to be \emph{globally-generated}, that is there exists a surjective sheaf-theoretical map $H^0 (\mathcal{L}) \otimes \mathcal{O}_\mani \to \mathcal{E}$, with $\dim H^0 (\mathcal{L}) = n+1 |m$. Then there exists a unique map up to isomorphism $\Phi_\mathcal{L} : \mani \to \proj {n|m}$ such that $\mathcal{E} = \Phi^\ast_\mathcal{L} (\mathcal{O}_{\proj {n|m}} (1))$ and such that, if $H^0 (\mathcal{L}) = \mbox{span}_\mathbb{C} \{ s_i | \xi_j \}$, then $s_i = \Phi^\ast_\mathcal{L} (X_i)$ and $\xi_j = \Phi^\ast_\mathcal{L} (\Theta_j)$ for $i=0, \ldots, n $ and $j= 1, \ldots, m$, where $X_i |\Theta_j$ are the generating sections of $ H^0(\mathcal{O}_{\proj {n|m}} (1))$, where we recall that $\mathcal{O}_{\proj {n|m}} (1) \defeq \pi^\ast \mathcal{O}_{\proj {n}} (1) = \pi^{-1} \mathcal{O}_{\proj n} (1) \otimes_{\pi^{-1} \mathcal{O}_{\proj n}} \mathcal{O}_{\proj {n|m}},$ see \cite{CN}, where invertible sheaves on projective superspaces are studied.

A very similar situation happens in the case of super Grassmannians, but instead of invertible sheaves one has to deal with locally-free sheaves of higher rank / vector bundles, in order to appropriately set up maps. Indeed, let $G = G(a|b, V)$ be a super Grassmannian, then it is has the following \emph{universal property} that characterizes the maps toward it \cite{CNR}: \vspace{5pt}\\
{\bf Universal Property:} for any supermanifold or superscheme $\mani$, any locally-free sheaf of $\stsheaf$-modules $\mathcal{E}$ of rank $a|b$ on $\mani$ and any vector superspace $V$ with a surjective sheaf-theoretical map
 $V\otimes \mathcal{O}_\mani \to \mathcal{E}$, then there exists a {\em unique} map $\Phi:\mani\to G(a|b, V)$ such that the inclusion 
$\mathcal{E}^\ast \to V^\ast\otimes \mathcal{O}_\mani$ is the pull-back of the inclusion $\mathcal{S}_G\to \mathcal{O}_G^{\oplus n|m}$ from the sequence 
\bear \label{tautsuper}
\xymatrix@R=1.5pt{
0 \ar[rr] && \mathcal{S}_G \ar[rr] && \mathcal{O}_{G}^{\oplus n|m} \ar[rr] && \widetilde{\mathcal{S}}_G^\ast \ar[rr] &&  0.
}  
\eear
where $\mathcal{S}_G$ is the tautological sheaf of the super Grassmannian.\vspace{.3cm}\\ 
\noindent Using the universal property above, we now explicitly show that there exists a map from a non-projected non-projective supermanifold of the family $\proj 2_{\omega } (\mathcal{F}_\mani)$, namely that one characterized by the decomposable fermionic sheaf $\mathcal{F}_\mani \defeq \Pi \mathcal{O}_{\proj 2} (-1)\oplus \Pi \mathcal{O}_{\proj 2} (-2)$, to a certain super Grassmannian, namely $G(2|2, \mathbb{C}^{12|12}).$ \\
For future use, we start giving in the following lemma the explicit form of the transition functions of this supermanifold in the case one chooses a decomposable fermionic sheaf as the one above.
\begin{lemma}[Transition functions] \label{transdec} Let $\proj {2}_\omega(\mathcal{F}_\mani)$ be the non-projected supermanifold with
$\mathcal{F}_\mani = \Pi \mathcal{O}_{\proj 2} (-1) \oplus \Pi \mathcal{O}_{\proj 2} (-2)$. Then, its transition functions take the following form:
\begin{align}\label{trans1}
& \mathcal{U}_{0} \cap \mathcal{U}_{1} : \qquad z_{10} = \frac{1}{z_{11}}, \qquad \quad z_{20} = \frac{z_{21}}{z_{11}} + \lambda \frac{\theta_{11} \theta_{21}}{(z_{11})^2};  & \qquad \theta_{10} = \frac{\theta_{11}}{z_{11}}, \qquad \theta_{20}  = \frac{\theta_{21}}
{(z_{11})^2};\nonumber \\  
& \mathcal{U}_{1} \cap \mathcal{U}_{2} : \qquad  z_{11} = \frac{z_{12}}{z_{22}} + \lambda \frac{\theta_{12} \theta_{22}}{(z_{22})^2}, \quad \qquad  z_{21} = \frac{1}{z_{22}}; & \qquad  \theta_{11} = \frac{\theta_{12}}{z_{22}}, \qquad \theta_{21}  = \frac{\theta_{22}}
{(z_{22})^2}; \nonumber\\  
& \mathcal{U}_{2} \cap \mathcal{U}_{0} : \qquad  z_{12} = \frac{1}{z_{20}}, \qquad  \quad z_{22} = \frac{z_{10}}{z_{20}} + \lambda \frac{\theta_{10} \theta_{20}}{(z_{20})^2}; & \qquad  \theta_{12} = \frac{\theta_{10}}{z_{10}},  \qquad \theta_{22} = 
\frac{\theta_{20}}{(z_{10})^2}.  
\end{align} 
\end{lemma} 
\begin{proof} The conclusion follows immediately from Theorem \ref{pi2}, taking into account the transition matrix for the given $\mathcal{F}_\mani$, that have the form $M=\left(\begin{array}{ll} \frac{1}{z_{01}} &0\\ 0& \frac{1}{z_{01}^2}\end{array}\right)$ on 
$\mathcal{U}_{0} \cap \mathcal{U}_{1}$ and similar form on the other two intersections of the fundamental open sets.  \end{proof}

\noindent Now we have to identify a suitable locally-free sheaf to set up the map into the super Grassmannian: a natural choice is given by the \emph{tangent sheaf} $\mathcal{T}_\mani$ of $\mani = \proj 2_\omega (\mathcal{F}_\mani)$ - which is obviously a rank $2|2$ locally-free sheaf in the case we are dealing with - and, possibly, its higher-symmetric powers $Sym^k \mathcal{T}_\mani:$ we will see that, in this case, $\mathcal{T}_\mani$ is actually enough and one does not need to resort to its higher symmetric products.  \\
In the following we will show that the vector superspace of global sections of the tangent sheaf $\mathcal{T}_\mani$, that is the 0-\v{C}ech cohomology space $H^0 (\mathcal{T}_\mani)$, is isomorphic to $\mathbb{C}^{12|12}$, and also, that one has a \emph{surjective} map $H^0 (\mathcal{T}_\mani) \otimes \mathcal{O}_\mani \rightarrow \mathcal{T}_\mani,$ that is the tangent sheaf $\mathcal{T}_\mani$ is \emph{globally-generated}. As in the universal property above, this implies that the choices of the tangent sheaf $\mathcal{T}_\mani$ for $\mathcal{E}$ and of $H^0 (\mathcal{T}_\mani)$ for $V$, lead to the existence of a (unique) map $\mani \to G(2|2, \mathbb{C}^{12|12}).$  \\
In order to prove the above statement, one needs to carefully study the tangent sheaf $\mathcal{T}_\mani$. We start considering the restriction of the tangent sheaf to the reduced manifold $\proj 2$, that is 
\bear
\mathcal{T}_\mani \lfloor_{\proj 2} = \mathcal{T}_\mani \oplus \stsheafred.
\eear
It is a general result that $\mathcal{T}_\mani \lfloor_{\manir} \cong \mathcal{T}_{\manir} \oplus \mathcal{F}_\mani^\ast$, see for example \cite{Manin} or \cite{NojaPhD}, anyway this result can be readily red off once one has the explicit form of the transition functions of the tangent sheaf. Indeed, using the chain rule and starting from the above lemma, with obvious notation, one finds:
\begin{align}
& \partial_{z_{10}} = - (z_{11})^2 \partial_{z_{11}} + [- z_{11} z_{21} + \theta_{11} \theta_{21}] \partial_{z_{21}} - \theta_{11} z_{11}\partial_{\theta_{11}} - 2 \theta_{21} z_{11} \partial_{\theta_{21}},  \nonumber \\
& \partial_{z_{20}} = z_{11} \partial_{z_{21}}, \nonumber \\
& \partial_{\theta_{10}} = - \theta_{21}\partial_{z_{21}} + z_{11} \partial_{\theta_{11}}, \nonumber \\
& \partial_{\theta_{20}} = z_{11} \theta_{11} \partial_{z_{21}} + (z_{11})^2 \partial_{\theta_{21}},
\end{align} 
so that the related Jacobian has the following matrix representation
\bear \label{jaco}
[\mbox{Jac}_{10}] = \left ( \begin{array}{cccc|cccc}
- (z_{11})^2 & & - z_{11} z_{21} + \theta_{11} \theta_{21} & &  - \theta_{11}z_{11} & & - 2 \theta_{21} z_{11} \\
0 & & z_{11} & & 0 & &  0 \\
\hline
0 & & - \theta_{21} & & z_{11} & & 0 \\
0 & & z_{11} \theta_{11} & & 0 & &  (z_{11})^2
\end{array}
\right ).
\eear
The transition functions in the other intersections can be found by $S_3$-symmetry.\\
\noindent We now recall that, having at disposal the structure sheaf of $\stsheaf$ of we can also form a sub-\emph{superscheme} of $\mani$ through the pair $(\proj 2, \mathcal{O}_{\mani}^{(2)} \defeq \slantone{\mathcal{O}_{\mani}}{\mathcal{J}_\mani^2} )$. We stress that this is not a supermanifold: indeed it fails to be locally isomorphic to any local model of the kind $\mathbb{C}^{m|n}$: more generally, it is locally isomorphic to an affine superscheme for some super ring. We call $\mani^{(2)}$ the superscheme defined by the pair $(\proj 2, \mathcal{O}^{(2)}_\mani )$ and we characterize its geometry in the following lemma.
\begin{lemma}[The Superscheme $\mani^{(2)}$] Let $\mani^{(2)}$ be the superscheme as above. Then $\mani $ is a projected scheme and its structure sheaf $\mathcal{O}^{(2)}_\mani$ is given by a locally-free sheaf of $\mathcal{O}_{\proj 2}$-algebras such that  
\bear
\mathcal{O}_{\mani}^{(2)} \cong \mathcal{O}_{\proj 2} \oplus \mathcal{F}_\mani. 
\eear
\end{lemma}
\begin{proof} it is enough to observe that the parity splitting of the structure sheaf reads $\mathcal{O}^{(2)}_\mani = \slantone{\mathcal{O}_{\mani,0}}{\mathcal{J}_\mani^2} \oplus \slantone{\mathcal{O}_{\mani,1 }}{\mathcal{J}_{\mani}^2}$, hence the defining short exact sequence for the even part reduces to an isomorphism $\mathcal{O}_{\mani,0}^{(2)} \cong \mathcal{O}_{\proj 2}$.The structure sheaf gets endowed with a structure of $\mathcal{O}_{\proj 2}$-module given by $\mathcal{O}_{\proj 2} \oplus \mathcal{F}_\mani$, that actually coincides with the parity splitting. \\
We observe that in the $\mathcal{O}_{\proj 2}$-algebra $ \mathcal{O}^{(2)}_\mani \cong \mathcal{O}_{\proj 2} \oplus \mathcal{F}_\mani$ the product $\mathcal{F}_\mani \otimes_{\mathcal{O}_{\proj 2}} \mathcal{F}_\mani \rightarrow \mathcal{O}_{\proj 2}$ is null. \end{proof}
\noindent Pushing the characterization of the tangent sheaf a little bit further, we have to study the geometry of tangent bundle $\mathcal{T}_\mani$ when restricted to the sub superscheme $\mani^{(2)}$. Once again it can be proved that the following general isomorphism holds true
\bear
\mathcal{T}_\mani \lfloor_{\mani^{(2)}} \cong \mathcal{T}_{\manir} \oplus \mathcal{E}nd (\mathcal{F}_\mani) \oplus \mathcal{F}_\mani^\ast \oplus (\mathcal{T}_{\manir} \otimes \mathcal{F}_\mani )
\eear
where the first two summands are the even part and the second two summands are the odd part of the sheaf.
In particular, in our case one gets:
\begin{lemma}[The Sheaf $\mathcal{T}_{\mani} \lfloor_{\mani^{(2)}}$] The sheaf $\mathcal{T}_\mani \lfloor_{\mani^{(2)}}$ is a locally-free of $\mathcal{O}_{\proj 2}$-module, moreover the following isomorphism holds
\bear
\mathcal{T}_\mani \lfloor_{\mani^{(2)}} \cong \slanttwo{\mathcal{T}_\mani}{\mathcal{J}^2 \mathcal{T}_\mani} \cong \mathcal{T}_\mani \lfloor_{\proj 2} \oplus \left ( \mathcal{T}_\mani \lfloor_{\proj 2} \otimes_{\mathcal{O}_{\proj 2}} \mathcal{F}_\mani \right )
\eear
\end{lemma}
\begin{proof} the claim is proved by computing 
\begin{align}
\mathcal{T}_\mani \lfloor_{\mani^{(2)}} \defeq \mathcal{T}_\mani \otimes_{\stsheaf} \mathcal{O}_{\mani^{(2)}} \cong \mathcal{T}_\mani \otimes_{\stsheaf} \left ( \mathcal{O}_{\proj 2} \oplus \mathcal{F}_\mani \right ) \cong \mathcal{T}_\mani \lfloor_{\proj 2} \oplus \left ( \mathcal{T}_\mani \lfloor_{\proj 2 } \otimes_{\mathcal{O}_{\proj 2}} \mathcal{F}_\mani \right ) 
\end{align}
where we have used that, since $\mathcal{F}_\mani$ is a locally-free sheaf of $\mathcal{O}_{\proj 2}$-module we have that $\mathcal{F}_\mani \cong \mathcal{F}_\mani \otimes_{\mathcal{O}_{\proj 2}} \mathcal{O}_{\proj 2}$. The first isomorphism is a standard result in modules theory (note we have suppressed the subscript $\mani$ in the sheaf of nilpotent element $\mathcal{J}_\mani$ for a better notation).
\end{proof}
\noindent For computational purposes, the sheaf $\mathcal{E} \lfloor_{\mani^{(2)}}$ can be made more explicit in its $\mathcal{O}_{\proj 2}$-module structure, indeed by making explicit its components one finds
\begin{align} \label{deco}
\slanttwo{\mathcal{T}_\mani}{\mathcal{J}_\mani^2\mathcal{T}_\mani} \cong \left [ \mathcal{T}_{\proj 2} \oplus \mathcal{O}_{\proj 2} (1) \oplus \mathcal{O}^{\oplus 2}_{\proj 2} \oplus \mathcal{O}_{\proj 2} (-1) \right] \oplus \Pi \left [ \mathcal{T}_{\proj 2} (-2) \oplus \mathcal{T}_{\proj 2} (-1) \oplus \mathcal{O}_{\proj 2} (2) \oplus \mathcal{O}_{\proj 2} (1) \right ].  
\end{align}
This decomposition will be useful once we have to compute the cohomology. \\

In order to compute the number of the global sections of the tangent sheaf of $\mani$, as to identify the supposed target super Grassmannian, we actually need one further sheaf, that we will study in the following lemma.
\begin{lemma}[The Sheaf $\mathcal{T}_\mani \otimes_{\stsheaf } \mathcal{J}^2_\mani$] The sheaf $\mathcal{T}_\mani \otimes_{\stsheaf } \mathcal{J}^2_\mani$ is isomorphic to ${\mathcal{J}^2} \mathcal{T}_\mani$. Moreover it is a locally-free sheaf of $\mathcal{O}_{\proj 2}$-modules and as such, it is isomorphic to $\mathcal{T}_\mani \lfloor_{\proj 2} (-3).$
\end{lemma}
\begin{proof} First of all we recall that $\mathcal{J}^2_\mani$ is a $\mathcal{O}_{\proj 2}$-module as it is killed by multiplication by $\mathcal{J}_\mani$. Moreover the tangent sheaf $\mathcal{T}_\mani$ is locally-free, and therefore it is flat, hence the functor $- \otimes_{\stsheaf} \mathcal{E}$ is exact. Let then us consider the short exact sequence
\bear
\xymatrix@R=1.5pt{
0 \ar[r] & \mathcal{J}_\mani^2 \ar[r] & \stsheaf \ar[r] & \slantone{\stsheaf}{\mathcal{J}^2_\mani} \ar[r] & 0. 
}  
\eear
By tensoring with $\mathcal{T}_\mani$ we get the short exact sequence
\bear
\xymatrix{
0 \ar[r] & \mathcal{J}_\mani^2\otimes_{\stsheaf} \mathcal{T}_\mani \ar[r] & \stsheaf \otimes_{\stsheaf} \mathcal{T}_\mani \cong \mathcal{T}_\mani \ar[r] & \slantone{\stsheaf}{\mathcal{J}^2_\mani} \otimes_{\stsheaf} \mathcal{T}_\mani \cong \slanttwo{\mathcal{T}_\mani}{\mathcal{J}^2 \mathcal{T}_\mani} \ar[r] & 0 \nonumber
}  
\eear
that implies that $\mathcal{J}^2\otimes_{\stsheaf} \mathcal{T}_\mani$ is indeed isomorphic to $\mathcal{J}^2_\mani \mathcal{T}_\mani$. Moreover we have that $\mathcal{J}_\mani^2 \cong Sym^2 \mathcal{F}_\mani$ and as such it is a $\mathcal{O}_{\proj 2}$-module, moreover, since $\mathcal{F}_\mani = \Pi \left ( \mathcal{O}_{\proj 2} (-1) \oplus \mathcal{O}_{\proj 2} (-2)\right )$, we have that $Sym^2 \mathcal{F}_\mani \cong \mathcal{O}_{\proj 2} (-3).$
\end{proof}
\noindent  We are now in the position to study the global sections of the tangent sheaf $\mathcal{T}_\mani$: the main tool we will use is the following exact sequence,
\bear \label{vectorb}
\xymatrix{
0 \ar[rr] && \mathcal{J}^2\mathcal{T}_\mani \ar[rr] && \mathcal{T}_\mani \ar[rr] && \slanttwo{\mathcal{T}_\mani}{\mathcal{J}_\mani^2 \mathcal{T}_\mani} \ar[rr] && 0 
}  
\eear
together with its long cohomology exact sequence. The previous lemmas together yields the following result.
\begin{lemma} The zeroth and the first cohomology groups of the sheaves $\mathcal{J}_\mani^2 \mathcal{T}_\mani$ and ${\mathcal{T}_\mani}/{\mathcal{J}^2_\mani \mathcal{T}_\mani}$ are given by
\begin{align}
& H^0 (\mathcal{J}_\mani^2 \mathcal{T}_\mani) = 0  & H^1 (\mathcal{J}^2 \mathcal{T}_\mani) = \mathbb{C}^{1|0} \\
& H^0 \left ({\mathcal{T}}/{\mathcal{J}_\mani^2 \mathcal{T}_\mani} \right ) = \mathbb{C}^{13|12}   & H^1 \left ( {\mathcal{T}_\mani}/{\mathcal{J}_\mani^2 \mathcal{T}_\mani} \right )  = 0.
\end{align}
\end{lemma}
\begin{proof} the result follows from a straightforward computation, once given the decomposition into direct sums of the sheaves above. 
\end{proof}
\noindent We are thus led to the following theorem, which is the main step toward the realization of an embedding into a super Grassmannian. 
\begin{theorem}[Global Sections of $\mathcal{T}_\mani$] \label{maintheorem}The tangent sheaf $\mathcal{T}_\mani$ of $\proj {2}_\omega (\mathcal{F}_\mani)$ has $12|12$ global sections. 
\end{theorem}
\begin{proof} Using the results of the previous lemma, the long exact cohomology sequence given by (\ref{vectorb}) reads
\bear
\xymatrix{
0 \ar[r] & H^0 (\mathcal{T}_\mani) \ar[r] & \mathbb{C}^{13|12} \ar[r]^\delta & \mathbb{C}^{1|0} \ar[r] & H^1(\mathcal{T}_\mani) \ar[r] & 0
}  
\eear
Therefore, since $H^1 (\mathcal{J}^2_\mani \mathcal{T}_\mani) \cong \mathbb{C}^{1|0}$ is $1$-dimensional, in order to prove surjectivity of the connection homomorphism $\delta : H^0 (\mathcal{T}_\mani / \mathcal{J}_\mani^2\mathcal{T}_\mani) \rightarrow H^1 (\mathcal{J}_\mani^2 \mathcal{T}_\mani)$, it is enough to show that it is not zero. To this end, we observe that in the decomposition (\ref{deco}) there is a term of the kind $\mathcal{O}^{\oplus2}_{\proj 2} \supset \mathcal{T}_\mani/{\mathcal{J}_\mani^2\mathcal{T}_\mani}$. It is easy to realize that the corresponding global sections $H^0 (\mathcal{O}_{\proj 2}\oplus \mathcal{O}_{\proj 2}) \subset H^0(\mathcal{E}/\mathcal{J}^2\mathcal{E})$ are of the form
\bear
s_1 = \theta_{1i} \otimes \partial_{\theta_{1i}} \qquad \qquad s_2 = \theta_{2i} \otimes \partial_{\theta_{2i}}
\eear
that we write multiplicatively as $ \theta_{1i} \partial_{\theta_{1i}}$ and $\theta_{2i} \partial_{\theta_{2i}}$ (both taken $\mbox{mod}\,\mathcal{J}^2_\mani$), indeed, changing coordinates, by means of the transformation rules obtained above, we get for example: 
\begin{align}
\theta_{10} \partial_{\theta_{10}} = \theta_{11} \partial_{\theta_{11}} -  \frac{\theta_{11} \theta_{21}}{z_{11}} \partial_{z_{21}} =  \theta_{11} \partial_{\theta_{11}}\, \mbox{mod}\, \mathcal{J}_\mani^2
\end{align} 
and, on the other hand we have that
\begin{align}
\left ( \theta_{10} \partial_{\theta_{10}} - \theta_{10} \partial_{\theta_{10}} \right )\Big \lfloor_{\mathcal{U}_0 \cap \mathcal{U}_1} = \frac{\theta_{11} \theta_{21}}{z_{11}} \partial_{z_{21}} \in \mathcal{J}_\mani^2\mathcal{T}_{\mani} (\mathcal{U}_0 \cap \mathcal{U}_1).
\end{align}
That is, we have that $\delta (s_1 ) \neq 0.$ Now, observing that $\frac{\theta_{11} \theta_{21}}{z_{11}} \partial_{z_{21}} = \frac{\theta_{11} \theta_{21}}{(z_{11})^2} \partial_{z_{20}}$, we conclude that
\bear
\left \{ \theta_{11} \partial_{\theta_{11}} - \theta_{10} \partial_{\theta_{10}} , \,  \theta_{12} \partial_{\theta_{12}} - \theta_{11} \partial_{\theta_{11}} , \, \theta_{10} \partial_{\theta_{10}} - \theta_{12} \partial_{\theta_{12}} \right \} \in Z^{1} (\mathcal{T}_{\proj {2}} (-3) ) \nonumber 
\eear
represents the same cocycle of $\mathcal{T}_{\proj 2} (-3)$ that determines the non-vanishing class $\omega \in H^1 (\mathcal{T}_{\proj {2}} (-3))$, as we have described early on. Observing that $H^1 (\mathcal{T}_\mani / \mathcal{J}_\mani^2 \mathcal{T}_\mani) \cong H^1 (\mathcal{T}_{\proj {2}} \otimes Sym^2 \mathcal{F}_\mani)$, we conclude that the connecting homomorphism is non-null, hence surjective. This splits the first part of the cohomology long exact sequence above in two pieces, in particular we have
\bear
\xymatrix{
0 \ar[rr] && H^0 (\mathcal{T}_\mani) \ar[rr] && \mathbb{C}^{13|12} \ar[rr]^\delta && \mathbb{C}^{1|0} \ar[rr] & & 0
}  
\eear
which proves that $H^0(\mathcal{T}_\mani) \cong \mathbb{C}^{12|12}$. 
\end{proof}
\noindent We are left to prove that the tangent sheaf $\mathcal{T}_\mani$ is actually globally-generated. This is achieved in the following lemma
\begin{lemma}[$\mathcal{T}_\mani$ is globally-generated] The tangent sheaf $\mathcal{T}_\mani$ of $\mani$ is such that the evaluation map $ev_{\mathcal{T}_\mani}: H^0 (\mathcal{T}_\mani)\otimes_{\stsheaf} \stsheaf \rightarrow \mathcal{T}_\mani$ is surjective. That is, $\mathcal{T}_\mani$ is globally-generated. 
\end{lemma}
\begin{proof} We start letting $W \defeq H^0 (\mathcal{O}_{\proj 2} \oplus \mathcal{O}_{\proj 2}) \subset H^0 (\mathcal{T}_\mani/ \mathcal{J}_\mani^2 \mathcal{T}_\mani)$ and $V$ be its complement into $H^0 (\mathcal{T}_\mani/ \mathcal{J}_\mani^2 \mathcal{T}_\mani)$, so that $V \oplus W = H^0 (\mathcal{T}_\mani/ \mathcal{J}_\mani^2 \mathcal{T}_\mani)$ and we call $U \defeq H^0 (\mathcal{T}_\mani)$. We have the following commutative diagram
\bear
\xymatrix{
 & \ker {\tilde i} \ar[d] \ar[r] & 0 \ar[d] \ar[r]  \ar[d] & 0 \ar[d]  \\
0 \ar[r] & U \cap W \ar[d]_{\tilde i} \ar[r] & W \ar[r] \ar[d]_{i_W} & \mathbb{C}^{1|0} \ar@{=}[d] \ar[r] & 0 \\
0 \ar[r] & U \ar[d] \ar[r] & V \oplus W \ar[d] \ar[r] & \mathbb{C}^{1|0} \ar[d] \ar[r] & 0 \\
& \coker {\tilde i} \ar[r] & V \ar[r] & 0 & 
}  
\eear
where $ \mathbb{C}^{1|0}$ correspond to $H^1 (\mathcal{J}_\mani^2\mathcal{T}_\mani)$, as computed above. Then, by snake lemma, we have an exact sequence
\bear
\xymatrix{
0 \ar[rr] && \coker {\tilde i} \ar[rr] && V \ar[rr] && 0   \\
}
\eear
therefore $\coker {\tilde i} \cong V$ and we have a surjection $ \xymatrix{U \ar@{>>}[r] &  V}$. In particular, since $H^0 (\mathcal{T}_\mani \lfloor_{\proj 2}) \subset V$ we have a surjective map $\psi : H^0(\mathcal{T}_\mani) \rightarrow H^0 (\mathcal{T}_\mani \lfloor_{\proj 2})$.
Now, let us consider the evaluation map $ev_{\mathcal{T}_\mani} : H^0 (\mathcal{T}_\mani) \otimes_{\stsheaf} \stsheaf \rightarrow \mathcal{T}_\mani$, which is a homomorphism of locally-free sheaves of $\stsheaf$-modules. Upon using Nakayama Lemma (see for example \cite{Vara}), it is enough to show that for all $x \in \proj 2$, the linear map 
\bear
\xymatrix@R=1.5pt{
ev_{\mathcal{T}_\mani}(x) :  H^0 (\mathcal{T}_\mani) \ar[rr] &&  \mathcal{T}_{\mani } (x)  \\
\; \;  \; \; s \ar@{|->}[rr] && s(x)
}  
\eear
that sends a global section $s$ to its evaluation $s(x)$ in $x \in \proj 2$ is surjective. This map can in turn be factored through $\psi $ as follows 
\bear
\xymatrix@R=1.5pt{
H^0 (\mathcal{T}_\mani) \ar[rr]^\psi &&  H^0(\mathcal{T}_\mani \lfloor_{\proj {2}}) \ar[rr] && \mathcal{T}_{\mani}(x) & x \in \proj 2\\
}.
\eear
Then, the first one has been just shown to be surjective, while the second one is well-known to be surjective as $\mathcal{T}_\mani \lfloor_{\proj 2}$ is a direct sum of globally-generated sheaves of $\mathcal{O}_{\proj 2}$-modules. This concludes the proof. \end{proof}
\noindent
The universal property, thus leads to the following 
\begin{theorem}[Map to $G(2|2, \mathcal{T}_\mani)$] There exists a unique map $\Phi_{\mathcal{T}_\mani} : \mani \longrightarrow G(2|2, \mathbb{C}^{12|12})$ up to isomorphism.
\end{theorem}
\noindent More can be said about this map, which is actually an \emph{embedding} of $\mani $ into $G(2|2, \mathbb{C}^{12|12})$: that is, it is an {injective} map and its differential $d\Phi_{\mathcal{T}_\mani}$ is injective as well. We prove this in a completely explicit fashion by realizing the actual embedding in a certain chart. \\

\noindent We explain the strategy to do this in a general setting: once one have a map into a super Grassmannian and a local basis $\{ e_1,\ldots,e_a | f_1,\ldots,f_b \}$ is fixed for $\mathcal{E}$ over some open set $\mathcal{U}$, then, over $\mathcal{U}$, the evaluation map $V\otimes \mathcal{O}_\mani\to \mathcal{E}$ is defined by a $(a|b)\times(n|m)$ 
matrix $M_\mathcal{U}$ with coefficients in $\mathcal{O}_\mani(\mathcal{U})$, and any reduction of  $M_\mathcal{U}$ into a standard form of type 
\begin{equation}\label{matrixstform}
\mathcal{Z}_{{I}} \defeq 
\left (
\begin{array}{ccc|ccc||ccc|ccc}
& &  & 1 \; & & & & & & & & \\
\;  & \; x_I \; & \; &  & \ddots & & & 0 & &\;  &\; \xi_I \; &\;  \\
& & & & & \; 1& & & & & & \\
\hline \hline
& & & & & & 1\; & & & & & \\
&\; \xi_I \; & & & 0& & &\ddots & & &\;  x_I \; & \\
& & & & & & & &\; 1 & & & 
\end{array}
\right ),
\end{equation}
by means of elementary row operations, is a local representation of the map $\Phi : \mani \rightarrow G(a|b, \mathbb{C}^{n|m})$. One can then easily verify injectivity and the injectivity of the differential of this map via this local representation, as to establish whether the map constitutes an embedding.\\ 

\noindent In order to do this, we need the explicit form of the global sections generating $\mathcal{T}_\mani$. Notice that to keep the discussion the most general possibile we will keep a \emph{parameter} $\lambda \in \mathbb{C}$ representing the cohomology class $\omega_\mani \in H^1 (\mathcal{T}_{\proj 2} (-3)) \cong \mathbb{C}$, which we recall to be the same $\lambda$ appearing in the transition functions provided by Theorem \ref{pi2}. 
\begin{theorem}[Generators of $H^0(\mathcal{T}_\mani)$] The tangent sheaf $\mathcal{T}_\mani$ of $\mani$ has $12|12$ global sections and in particular, in the local chart $\, \mathcal{U}_0$, a basis for $H^0(\mathcal{T}_\mani)$ is given by 
$\mbox{\emph{span}}_\mathbb{C} \{ \mathcal{V}_1, \ldots, \mathcal{V}_{12} |\, \Xi_1 , \ldots, \Xi_{12} \}$, where
\begin{align}
& \mathcal{V}_1 = \partial_{z_{10}}, \qquad \mathcal{V}_2 = \partial_{z_{20}}, \qquad \mathcal{V}_3 = z_{20} \partial_{z_{10}}, \qquad \mathcal{V}_4 = z_{10} \partial_{z_{20}}, \qquad \mathcal{V}_5 = z_{10} \partial_{z_{10}} - z_{20} \partial_{z_{20}}, \nonumber \\
& \mathcal{V}_6 = \theta_{10} \partial_{\theta_{20}}, \qquad \mathcal{V}_{7} = z_{10}\theta_{10}\partial_{\theta_{20}}, \qquad \mathcal{V}_{8} = z_{20} \theta_{10} \partial_{\theta_{20}}, \qquad \nonumber \\
& \mathcal{V}_9 = \theta_{10} \partial_{\theta_{10}} + z_{20} \partial_{z_{20}}, \qquad \mathcal{V}_{10} = \theta_{20} \partial_{\theta_{20}} + z_{20}\partial_{z_{20}}, \nonumber \\
& \mathcal{V}_{11} = (z_{10})^2\partial_{z_{10}} + (z_{10} z_{20} + \lambda \theta_{10} \theta_{20}) \partial_{z_{20}} + z_{10} \theta_{10} \partial_{\theta_{10}} + 2 z_{1 0}\theta_{20} \partial_{\theta_{20}}, \nonumber \\
& \mathcal{V}_{12} = (z_{10}z_{20} - \lambda \theta_{10} \theta_{20})\partial_{z_{10}} + (z_{20})^2 \partial_{z_{20}} + z_{20} \theta_{10}\partial_{\theta_{10}} + 2 z_{20} \theta_{20} \partial_{\theta_{20}}, \nonumber \\ 
\nonumber \\
& \Xi_1 =  \partial_{\theta_{10}}, \qquad \Xi_2 = \partial_{\theta_{20}}, \qquad \Xi_3 = \theta_{10}\partial_{z_{10}}, \qquad \Xi_{4} = \theta_{10} \partial_{z_{20}}, \qquad \Xi_5 = z_{10} \partial_{\theta_{20}}, \qquad \Xi_6 = z_{20} \partial_{\theta_{20}}, 
\nonumber \\
& \Xi_{7} = (z_{10})^2 \partial_{\theta_{20}} - \lambda z_{10} \theta_{10} \partial_{z_{20}}, \qquad \Xi_{8} = (z_{20})^2 \partial_{\theta_{20}} + \lambda z_{20} \theta_{10} \partial_{z_{10}}, \nonumber \\
& \Xi_9 = z_{10} \partial_{\theta_{10}} + \lambda \theta_{20} \partial_{z_{20}}, \qquad \Xi_{10} = - z_{20} \partial_{\theta_{10}} + \lambda \theta_{20} \partial_{z_{10}}, \nonumber \\
& \Xi_{11} = z_{10} \theta_{10} \partial_{z_{10}} + z_{20} \theta_{10} \partial_{z_{20}} + 2 \theta_{10} \theta_{20} \partial_{\theta_{20}}, \nonumber \\
& \Xi_{12} = (z_{10} z_{20} - \lambda \theta_{10}\theta_{20}) \partial_{\theta_{20}} - \lambda z_{20} \theta_{10} \partial_{z_{20}}, 
\end{align}
where $\lambda \in \mathbb{C}$ is a complex number representing the cohomology class $H^1 (\mathcal{T}_\proj 2 (-3)) \cong \mathbb{C}$.  
\end{theorem}
\begin{proof} The theorem is proved by evaluating the 0-\v{C}ech cohomology group of $\mathcal{T}_\mani$, by means of a computation in charts.\end{proof}
\noindent The embedding is explicitly realized through the following 
Now, following what explained above, the coefficients of the expansion are mapped into $12|12$ columns, so that the resulting matrix is a super Grassmannian of the kind $G(2|2, \mathbb{C}^{12|12})$, represented in a certain super big-cell. The full super Grassmannian is then reconstructed via its transition functions, as explained in the previous section.\\
\noindent In our particular case, the global sections lead to an image into $G(2|2, \mathbb{C}^{12|12})$ as follows:
\bear
\Phi_{\mathcal{T}_\mani} (\mani) = \left ( 
\begin{array}{cc|ccc|cc|ccc}
1 & 0 & \quad & A_{1\times10} & \quad & 0 & 0 & \quad & B_{1 \times 10} & \quad   \\
0 & 1 & \quad & A_{2\times 10}& \quad & 0 & 0 & \quad & B_{2 \times 10} & \quad \\  
\hline
0 & 0 & \quad & C_{1 \times 10 } & \quad & 1 & 0 & \quad  & D_{1\times 10 }& \quad  \\
0 & 0 & \quad & C_{2 \times 10 } & \quad & 0 & 1 & \quad & D_{2\times 10} & \quad 
\end{array}
\right ),
\eear
where we have highlighted the super big-cell singled out by the four global sections $\{ \mathcal{V}_1 = \partial_{z_1}, \mathcal{V}_2 = \partial_{z_2}, \Xi_1 = \partial_{\theta_1}, \Xi_2 = \partial_{\theta_2} \}$ in the chart $\mathcal{U}_0$ and the 
$A_{i \times 10}, B_{i \times 10}, C_{i \times 10}, D_{i \times 10}$ for $i = 1,2$, make up four $2 \times 10$ matrices:
\begin{align}
& A \defeq \left ( \begin{array}{c}
A_{1\times 10}\\
A_{2\times 10} 
\end{array}
\right )
= \left ( \begin{array}{cccccccccccccccccccc}
z_2 & & 0     & & z_1   & &  0            & & 0                   & & 0                   & & 0            & & 0             & & z_1^2                                                 & & z_1z_2 - \lambda \theta_1 \theta_2 \\
0 	   & &  z_1 & & - z_2 & & 0            & & 0                   & & 0                   & & z_2        & &  z_2         & & z_1 z_2 + \lambda \theta_1 \theta_2 & & z_2^2                                   
\end{array}
\right ), \nonumber \\ \nonumber
\\
& B \defeq \left ( \begin{array}{c}
B_{1\times 10}\\
B_{2\times 10} 
\end{array}
\right )
= \left ( \begin{array}{cccccccccccccccccccc}
 \theta_1 & & 0             && 0    & &0    & & 0                                  & & \lambda z_2 \theta_1 & &0                         & & \lambda \theta_2 & & z_1\theta_1      & & 0    \\
 0            &  &\theta_1  && 0    & & 0    & & -\lambda z_1 \theta_1 & & 0                                 & & \lambda \theta_2 & & 0                        & & z_2 \theta_1     & &- \lambda z_2         
\end{array}
\right ), \nonumber \\
\nonumber \\
& C \defeq \left ( \begin{array}{c}
C_{1\times 10}\\
C_{2\times 10} 
\end{array}
\right ) 
= \left ( \begin{array}{cccccccccccccccccccc}
0     & & 0     & & 0       & &  0            & & 0                   & & 0                   & & \theta_1 & & 0             & & z_1\theta_1                                       & & z_2 \theta_1                                   \\
 0     & & 0     & & 0       &  & \theta_1 & & z_1 \theta_1 & & z_2 \theta_1 & & 0             & & \theta_2 & & 2z_1 \theta_2                                    & & 2z_2 \theta_2                                \end{array}
\right ), \nonumber \\
\nonumber \\
& D \defeq \left ( \begin{array}{c}
D_{1\times 10}\\
D_{2\times 10} 
\end{array}
\right )
= \left ( \begin{array}{cccccccccccccccccccc}
 0            & & 0    & & 0      & &  0                                  & &  0            & & 0                   & & z_1                      & & -z_2                   & & 0 & & 0       \\
  0            & & 0            & & z_1 & & z_2 & & z_1^2                           & & z_2^2                               & & 0                  & &  0                         & & 2\theta_1 \theta_2 & & z_1 z_2 - \lambda \theta_1 \theta_2
 \end{array}
\right ),
\end{align}
where the subscript referring to the chart $\mathcal{U}_0$ of $\mani$ has been suppressed for readability purpose. One can then confirm that the map $\Phi_{\mathcal{T}_\mani}$ is indeed an embedding via this explicit expression.  
\begin{theorem} Let $\proj {2}_\omega (\mathcal{F}_\mani)$ be the non-projected supermanifold endowed with a fermionic sheaf $\mathcal{F}_\mani \defeq \Pi \mathcal{O}_{\proj 2} (-1) \oplus \Pi \mathcal{O}_{\proj 2} (-2)$. Then the map 
$i : \proj {2}_\omega (\mathcal{F}_\mani) \rightarrow G(2|2, \mathbb{C}^{12|12})$ is an embedding of supermanifolds.
\end{theorem}
\begin{proof} One checks from the expressions above that the map is injective on the geometric points, that is on $\proj 2$, and that its super differential is injective. This can be checked, for example, by representing the super differential 
as a $4 \times 80 $ matrix, where the four $1\times 80$ rows are given by the derivatives of a row vector $(A_{i \times 10}, B_{i \times 10 }, C_{i \times 10}, D_{i \times 10})$ with respect to 
$\partial_{z_1}, \partial_{z_2}, \partial_{\theta_1}, \partial_{\theta_2}.$ The resulting matrix has indeed rank $4$. \end{proof}
\noindent It is fair to say, by the way, that one can simplify the proof and avoid cumbersome computation, by considering just a subset of the global sections found above in order to prove global generation and injectivity of the differential. For example, the subset of $H^0 (\mathcal{T}_\mani)$ given by the sections  
\bear
S \defeq \left \{ \mathcal{V}_1, \mathcal{V}_2, \mathcal{V}_5, \mathcal{V}_{9}-\mathcal{V}_{10}, \Xi_1 , \Xi_2 \right \} \subset H^{0} (\mathcal{T}_\mani). 
\eear
does the job. Indeed, these sections make up a sub-matrix of the $12|12 \times 4|4$ matrix given, having columns given by coordinates of the global sections with respect to the basis $\partial_{z_1}, \partial_{z_2}, \partial_{\theta_1}, \partial_{\theta_2}$ in the chart $\mathcal{U}_0$ as above. Writing the columns in a suitable order, one gets
\bear
i (S) = \left (\begin{array}{c|cccc|cc}
& \mathcal{V}_9-\mathcal{V}_{10} & \mathcal{V}_5 & \mathcal{V}_1 & \mathcal{V}_2 & \Xi_1 & \Xi_2 \\
\hline
\partial_{z_1} & 0 & z_1 & 1 & 0 & 0 & 0 \\
\partial_{z_2} & 0 & -z_2 & 0 & 1 & 0 & 0 \\
\hline
\partial_{\theta_1} & \theta_1 & 0 & 0 & 0 & 1 & 0\\
\partial_{\theta_2} & - \theta_2 & 0 & 0 & 0 & 0 & 1
\end{array}
\right ).
\eear
This is a \emph{linear} embedding of $\mathcal{U}_0$ into a super big-cell of the super Grassmannian: which proves both global generation and injectivity at the level of the differential over $\mathcal{U}_0$ at once. Also, by symmetry, or analogously by the homogeneity of $\mani$ and $\mathcal{T}_\mani$ with respect to the action of $PGL(3)$, the same result holds true over $\mathcal{U}_1$ and $\mathcal{U}_2$ as well. 

\section*{Acknowledgements} 
\noindent This research is original and has a financial support of the Universit\`a del Piemonte Orientale. 
(Fondi Ricerca Locale).

\bibliographystyle{amsplain}

\end{document}